\documentclass[11pt,a4paper]{amsart}

\usepackage[utf8]{inputenc}
\usepackage{enumitem,a4,titletoc}
\usepackage{bbm,amssymb,mathtools,mathrsfs}
\usepackage[usenames,dvipsnames]{xcolor}
\usepackage{tikz}
\usetikzlibrary{matrix,arrows}
\usepackage[titletoc]{appendix}
\usepackage[colorlinks,linkcolor=BrickRed,citecolor=OliveGreen,urlcolor=black,hypertexnames=true]{hyperref}
\usepackage[margin=2.5cm]{geometry}

\DeclareMathOperator{\gm}{GM}
\DeclareMathOperator{\usa}{USA}

\DeclareMathOperator{\im}{im}
\DeclareMathOperator{\kar}{char}

\DeclareMathOperator{\Sym}{Sym}
\DeclareMathOperator{\tr}{tr}
\DeclareMathOperator{\pf}{pf}
\DeclareMathOperator{\sgn}{sgn}
\DeclareMathOperator{\Tr}{Tr}

\DeclareMathOperator{\GL}{GL}
\DeclareMathOperator{\End}{End}
\DeclareMathOperator{\Sper}{Sper}

\DeclareMathOperator{\opm}{M}
\DeclareMathOperator{\diag}{diag}
\newcommand{\N}{\mathbb{N}}
\newcommand{\Z}{\mathbb{Z}}
\newcommand{\Q}{\mathbb{Q}}
\newcommand{\R}{\mathbb{R}}
\newcommand{\C}{\mathbb{C}}
\newcommand{\HH}{\mathbb{H}}
\newcommand{\TT}{\mathbb{T}}
\newcommand{\cA}{\mathcal{A}}
\newcommand{\cB}{\mathcal{B}}
\newcommand{\cC}{\mathcal{C}}

\newcommand{\cJ}{\mathcal{J}}
\newcommand{\JJ}{\mathscr{J}}

\newcommand{\cN}{\mathcal{N}}

\newcommand{\cR}{\mathcal{R}}

\newcommand{\cO}{\mathcal{O}}

\newcommand{\cQ}{\mathfrak{Q}}
\newcommand{\cT}{\mathfrak{T}}

\newcommand{\ti}{\mathsf t}
\newcommand{\si}{\mathsf s}
\newcommand{\ve}{\varepsilon}

\newcommand{\Grad}{H}

\newcommand{\ulx}{\boldsymbol{x}}
\newcommand{\ulxi}{\boldsymbol{\xi}}

\newcommand{\Langle}{\mathop{<}\!}
\newcommand{\Rangle}{\!\mathop{>}}
\newcommand{\px}{F\!\Langle \ulx,\ulx^*\Rangle}

\newcommand{\bOmega}{\mbox{$\hspace{0.26em}\rule[0.04em]{0.04em}{0.65em}\hspace{-0.26em}\Omega\hspace{-0.27em}\rule[0.04em]{0.04em}{0.64em}\hspace{0.27em}$}}
\newcommand{\TR}{\TT_n}
\newcommand{\STR}{\Sym\mathbb{T}_n}
\newcommand{\TP}{\bOmega_n}
\newcommand{\TPc}{\Omega_n}
\newcommand{\rey}{\cR_n}
\newcommand{\mat}{\opm_n(\R)}
\newcommand{\matpoly}{\opm_n(\R[\ulxi])}
\newcommand{\smatpoly}{\Sym\opm_n(\R[\ulxi])}
\newcommand{\matrat}{\opm_n(\R(\ulxi))}
\newcommand{\ort}{\operatorname{O}_n(\R)}
\newcommand{\sort}{\operatorname{SO}_n(\R)}
\newcommand{\cen}{T_n}
\newcommand{\GM}{\gm_n}
\newcommand{\SGM}{\Sym\gm_n}
\newcommand{\ceng}{C_n}
\newcommand{\USA}{\usa_n}

\newcommand{\free}{\mathbb{T}}
\newcommand{\mx}{\Langle \ulx,\ulx^*\Rangle}

\makeatletter
\def\moverlay{\mathpalette\mov@rlay}
\def\mov@rlay#1#2{\leavevmode\vtop{
		\baselineskip\z@skip \lineskiplimit-\maxdimen
		\ialign{\hfil$#1##$\hfil\cr#2\crcr}}}
\makeatother

\newtheorem{thm}{Theorem}[section]
\newtheorem{lem}[thm]{Lemma}
\newtheorem{cor}[thm]{Corollary}
\newtheorem{prop}[thm]{Proposition}
\newtheorem{conj}[thm]{Conjecture}
\newtheorem*{prop*}{Proposition}
\newtheorem{thmA}{Theorem}

\newtheorem{thma}{Theorem}

\newtheorem{cora}{Corollary}

\theoremstyle{definition}

\newtheorem{exa}[thm]{Example}
\newtheorem*{exa*}{Example}

\theoremstyle{remark}
\newtheorem{rem}[thm]{Remark}

\numberwithin{equation}{section}

\begin{document}
	
\setcounter{tocdepth}{3}
\contentsmargin{2.55em} 
\dottedcontents{section}[3.8em]{}{2.3em}{.4pc} 
\dottedcontents{subsection}[6.1em]{}{3.2em}{.4pc}
\dottedcontents{subsubsection}[8.4em]{}{4.1em}{.4pc}

\makeatletter
\newcommand{\mycontentsbox}{%

	{\centerline{NOT FOR PUBLICATION}
		\addtolength{\parskip}{-2.3pt}
		\tableofcontents}}
\def\enddoc@text{\ifx\@empty\@translators \else\@settranslators\fi
	\ifx\@empty\addresses \else\@setaddresses\fi
	\newpage\mycontentsbox\newpage\printindex}
\makeatother

\setcounter{page}{1}

\title[Positive trace polynomials]{Positive trace polynomials and \\ the universal Procesi-Schacher conjecture}

\author[I. Klep]{Igor Klep${}^1$}
\address{Igor Klep, Department of Mathematics, University of Auckland}
\email{igor.klep@auckland.ac.nz}
\thanks{${}^1$Supported by the Marsden Fund Council of the Royal Society of New Zealand. Partially supported
by the Slovenian Research Agency grants P1-0222, L1-6722 and J1-8132.}

\author[\v{S}. \v{S}penko]{\v{S}pela \v{S}penko${}^2$}
\thanks{${}^2$The second author is a FWO $[$PEGASUS$]^2$ Marie Sk\l odowska-Curie fellow at the Free University of Brussels
(funded by the European Union Horizon 2020 research and innovation
programme under the Marie Sk\l odowska-Curie grant agreement
No 665501 with the Research Foundation Flanders (FWO)). During part of this work she was also a postdoc with Sue Sierra at the University of Edinburgh}
\address{\v{S}pela \v{S}penko, Departement Wiskunde, Vrije Universiteit Brussel}
\email{spela.spenko@vub.ac.be}

\author[J. Vol\v{c}i\v{c}]{Jurij Vol\v{c}i\v{c}${}^3$}
\address{Jurij Vol\v{c}i\v{c}, Department of Mathematics, Ben-Gurion University of the Negev}
\email{volcic@post.bgu.ac.il}
\thanks{${}^3$Supported by the University of Auckland Doctoral Scholarship and the Deutsche Forschungsgemeinschaft (DFG) Grant No. SCHW 1723/1-1.}

\date{\today}

%%%%%%%%%%%%%%%%%%%%%%%%%%%%%%%%%%%%%%%%%%%%%%%%%%%%%%%%%%%%%%%%%%%%%%%%%%%%

\begin{abstract}
Positivstellens{\"a}tze are fundamental results in real algebraic geometry providing algebraic certificates for positivity of polynomials on semialgebraic sets. In this article Positivstellens\"atze for trace polynomials positive on semialgebraic sets of $n\times n$ matrices are provided. A Krivine-Stengle-type Positivstellensatz is proved characterizing trace polynomials nonnegative on a general semialgebraic set $K$ using weighted sums of hermitian squares with denominators. The weights in these certificates are obtained from generators of $K$ {\it and} traces of hermitian squares. For compact semialgebraic sets $K$ Schm\"udgen- and Putinar-type Positivstellens\"atze are obtained: every trace polynomial positive on $K$ has a sum of hermitian squares decomposition with weights and without denominators. The methods employed are 
inspired by 
invariant theory, classical real algebraic geometry and functional analysis.

Procesi and Schacher in 1976 developed a theory of orderings and positivity on central simple algebras with involution and posed a Hilbert's 17th problem for a universal central simple algebra of degree $n$: is every totally positive element a sum of hermitian squares? They gave an affirmative answer for $n=2$. In this paper a negative answer for $n=3$ is presented. Consequently, 
including
traces of hermitian squares as weights 
in the Positivstellens\"atze
is indispensable.\looseness=-1
\end{abstract}

%%%%%%%%%%%%%%%%%%%%%%%%%%%%%%%%%%%%%%%%%%%%%%%%%%%%%%%%%%%%%%%%%%%%%%%%%%%%

\maketitle

%%%%%%%%%%%%%%%%%%%%%%%%%%%%%%%%%%%%%%%%%%%%%%%%%%%%%%%%%%%%%%%%%%%%%%%%%%%%
%%%%%%%%%%%%%%%%%%%%%%%%%%%%%%%%%%%%%%%%%%%%%%%%%%%%%%%%%%%%%%%%%%%%%%%%%%%%

\section{Introduction}\label{s:intro}

Positivstellensätze are pillars of modern real algebraic geometry \cite{BCR,PD,Mar,Sch1}. A Positivstellensatz is an algebraic certificate for a real polynomial to be positive on a set described by polynomial inequalities. For a finite set $S\subset\R[\ulxi]=\R[\xi_1,\dots,\xi_g]$ let $K_S$ denote the semialgebraic set of points $\alpha\in\R^g$ for which $s(\alpha)\ge0$ for all $s\in S$. The most fundamental result here is the Krivine-Stengle Positivstellensatz (see e.g. \cite[Theorem 2.2.1]{Mar}), which characterizes polynomials that are positive on $K_S$ as weighted sums of squares with denominators, where weights are products of elements in $S$. This theorem is the real  analog of Hilbert's Nullstellensatz and a far-reaching generalization of Artin's solution to Hilbert's 17th problem. If the set $K_S$ is compact, a simpler description of strict positivity on $K_S$ is given by Schmüdgen's Positivstellensatz \cite{Schm}. If moreover $S$ generates an archimedean quadratic module, then Putinar's Positivstellensatz \cite{Put} presents an even simpler form of strictly positive polynomials on $K_S$. The latter leads to a variety of applications of real algebraic geometry via semidefinite programming \cite{WSV,BPT}
 in several areas of applied mathematics and engineering. By adapting the notion of a quadratic module and a preordering to $\matpoly$, many of the results described above extend to matrix polynomials \cite{GR,SH,Cim}.

Positivstellensätze are also key in noncommutative real algebraic geometry \cite{HP,Schm1,Oza}, where the theory essentially divides into two parts between which there is increasing synergy. The dimension-free branch started with Helton's theorem characterizing free 
noncommutative
polynomials, which are positive semidefinite on all matrices of all sizes, as sums of hermitian squares \cite{Hel}. This principal result was followed by various Positivstellensätze in a free algebra \cite{HM,HKM,KVV}, often with cleaner statements or stronger conclusions than their commutative counterparts. These dimension-free techniques are also applied to positivity in operator algebras \cite{NT,Oza1} and free probability \cite{GS}. Trace positivity of free polynomials presents the algebraic aspect of the renowned Connes' embedding conjecture \cite{KS,Oza}. In addition to convex optimization \cite{BPT}, free positivity certificates frequently appear in quantum information theory \cite{NC} and control theory \cite{BEFB}. On the other hand, the dimension-dependent branch is less developed. Here the main tools come from the theory of quadratic forms, polynomial identities and central simple algebras with involution \cite{KMRT,Row2,AU0}. 
A fundamental result in this context, 
a Hilbert's 17th problem, was solved by Procesi and Schacher \cite{PS}:
totally positive elements in a central simple algebra with a positive involution are weighted sums of hermitian squares, and the weights arise as traces of hermitian squares. Analogous conclusions hold for trace positive polynomials \cite{Kle}.
The basic problem here is whether
the traces of hermitian squares are actually needed; cf.~\cite{KU,AU,SS}.

We next outline the contributions of this paper. Let $\free$ be the {\bf free trace ring}, i.e., the $\R$-algebra with involution $*$ generated by noncommuting variables $x_1,\dots,x_g$ and symmetric commuting variables $\Tr(w)$ for words $w$ in $x_j,x_j^*$ satisfying $\Tr(w_1w_2)=\Tr(w_2w_1)$ and $\Tr(w^*)=\Tr(w)$. Let $\Sym \free$ be the subspace of symmetric elements, $T$ the center of $\free$ and $\Tr:\free\to T$ the natural $T$-linear map. For a fixed $n\in\N$, the evaluation of $\free$ at $X\in\mat^g$ is defined by $x_j\mapsto X_j$, $x_j^*\mapsto X_j^{\ti}$ and $\Tr(w)\mapsto \tr(w(X))$. 

\begin{exa*}
Consider $f=5 \Tr(x_1x_1^*)-2 \Tr(x_1)(x_1+x_1^*)\in\free$. We claim that
$f$ is positive (semidefinite) on $\opm_2(\R)$. For $X\in\opm_2(\R)$ write
$$H_1=X-X^{\ti},\qquad H_2=XX^{\ti}-X^{\ti}X,\qquad H_3=X^2-2XX^{\ti}+2X^{\ti}X-(X^{\ti})^2 .$$
If $H_1$ is invertible, then one can check (see Example \ref{exa:intro} for details) that
$$f(X)=\frac52 H_1H_1^{\ti}+\frac12 H_1^{-1}H_2H_2^{\ti}H_1^{-\ti}+\frac12 H_1^{-1}H_3H_3^{\ti}H_1^{-\ti}$$
and hence $f(X)\succeq0$, so $f\succeq0$ on $\opm_2(\R)$ by continuity. On the other hand, $f$ is not positive on $\opm_3(\R)$:
$$\renewcommand*{\arraystretch}{.8}
f\Big(\!\begin{pmatrix}
2 & 0 & 0\\ 0 & 1 & 0 \\ 0 & 0 & 1
\end{pmatrix}
\!\Big) = 
5 \Tr\Big(\!\begin{pmatrix}
4 & 0 & 0\\ 0 & 1 & 0 \\ 0 & 0 & 1
\end{pmatrix}\!\Big) I_3
- 2 \Tr \Big(\!\begin{pmatrix}
2 & 0 & 0\\ 0 & 1 & 0 \\ 0 & 0 & 1
\end{pmatrix}\!\Big) \begin{pmatrix}
4 & 0 & 0\\ 0 & 1 & 0 \\ 0 & 0 & 1
\end{pmatrix}
=
\begin{pmatrix}
 -2 & 0 & 0 \\
 0 & 14 & 0 \\
 0 & 0 & 14  
 \end{pmatrix}\not\succeq0.\renewcommand*{\arraystretch}{1}
$$
In the rest of the paper we will develop a  systematic theory for positivity of trace polynomials.
\end{exa*}

For $S\subset\Sym\free$ let
$$K_S=\left\{X\in\mat^g\colon s(X)\succeq0\ \forall s\in S \right\}.$$
If $S$ is finite, then $K_S$ is the {\bf semialgebraic set} described by $S$. A set $\cQ\subset \Sym\free$ is a {\bf cyclic quadratic module} if
$$1\in\cQ,\quad \cQ+\cQ\subseteq \cQ,
\quad h\cQ h^*\subseteq \cQ\ \ \forall h\in\free,
\quad \Tr(\cQ)\subset \cQ.$$
A cyclic quadratic module $\cT$ is a {\bf cyclic preordering} if $T\cap \cT$ is closed under multiplication.

\begin{prop*}
If $\cT\subset\free$ is a cyclic preordering, then $f|_{K_{\cT}}\succeq0$ for every $f\in\cT$.
\end{prop*}

The converse of this simple proposition fails in general, but the next noncommutative version of the Krivine-Stengle Positivstellensatz 
uses
cyclic preorderings to
describe noncommutative polynomials positive semidefinite on a semialgebraic set $K_S$.

\begin{thma}\label{ta:free}
Let $S\cup\{f\}\subset\Sym\free$ be finite and let $\cT$ be the smallest
cyclic preordering containing $S$. Then $f|_{K_S}\succeq0$ if and only if
\begin{equation*}
(t_1f)|_{\mat^g}=(f^{2k}+t_2)|_{\mat^g}\qquad \text{and} \qquad (ft_1)|_{\mat^g}=(t_1f)|_{\mat^g}
\end{equation*}
for some $k\in\N$ and $t_1,t_2\in\cT$.
\end{thma}

See Theorem \ref{ta:ks} below for an extended version in a slightly different language. The existence of trace identities on $n\times n$ matrices  suggests that the problem of positivity on $n\times n$ matrices should be treated in an appropriate quotient of $\free$, which we describe next.

Let $\TR$ be the {\bf trace ring of generic $n\times n$ matrices}, i.e., the $\R$-algebra generated by generic matrices $\Xi_1,\dots,\Xi_g$, their transposes and traces of their products. Here $\Xi_j=(\xi_{j\imath\jmath})_{\imath\jmath}$ is an $n\times n$ matrix whose entries are independent commuting variables. The $\R$-subalgebra $\GM\subset\TR$ generated by $\Xi_j,\Xi_j^{\ti}$ is called the {\bf ring of generic $n\times n$ matrices} and has a central role in the theory of polynomial identities \cite{Pro,Row2}. The ring of central quotients of $\GM$ is the {\bf universal central simple algebra with orthogonal involution} of degree $n$, denoted $\USA$. The ring $\TR$ also has a geometric interpretation. Let the orthogonal group $\ort$ act on $\mat^g$ by simultaneous conjugation. If $\matpoly$ is identified with polynomial maps $\mat^g\to\mat$, then $\TR$ is the ring of polynomial $\ort$-concomitants in $\matpoly$, i.e., equivariant maps with respect to the $\ort$-action \cite{Pro}. If $\ceng,\cen,Z_n$ are the centers of $\GM,\TR,\USA$, respectively, and $\cR$ is the ``averaging'' Reynolds operator for the $\ort$-action, then we have the following diagram.

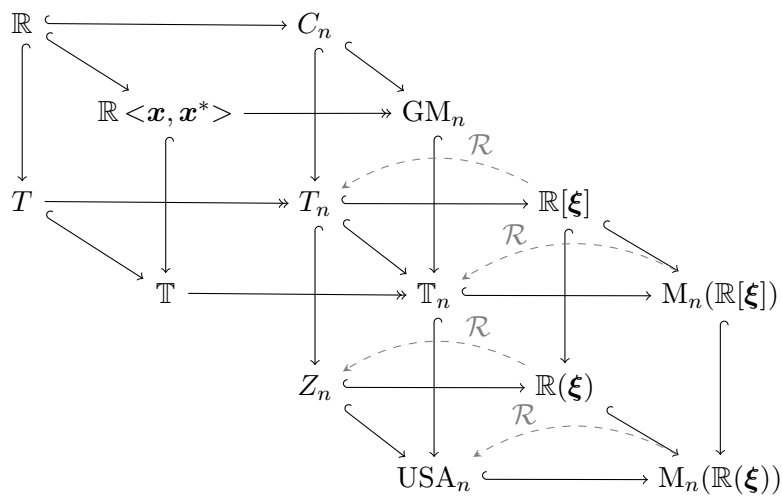
\begin{center}
\begin{tikzpicture}
\matrix (m) [matrix of math nodes,row sep=1.5em,column sep=1.5em,minimum width=1.5em]
{
	\R & & \ceng & & & \\
	& \R\mx & & \GM & & \\
	T & & \cen & & \R[\ulxi] & \\
	& \free & & \TR & & \matpoly \\
	& & Z_n & & \R(\ulxi) & \\
	& & & \USA & & \opm_n(\R(\ulxi)) \\};
\path[-stealth]
(m-1-1) edge [right hook->] (m-1-3)
(m-2-2) edge [->>] (m-2-4)
(m-3-1) edge [->>] (m-3-3)
(m-3-3) edge [right hook->] (m-3-5)
(m-4-2) edge [->>] (m-4-4)
(m-4-4) edge [right hook->] (m-4-6)
(m-5-3) edge [right hook->] (m-5-5)
(m-6-4) edge [right hook->] (m-6-6)

(m-1-1) edge [right hook->] (m-3-1)
(m-1-3) edge [right hook->] (m-3-3)
(m-3-3) edge [right hook->] (m-5-3)
(m-3-5) edge [right hook->] (m-5-5)
(m-2-2) edge [right hook->] (m-4-2)
(m-2-4) edge [right hook->] (m-4-4)
(m-4-4) edge [right hook->] (m-6-4)
(m-4-6) edge [right hook->] (m-6-6)

(m-1-1) edge [right hook->] (m-2-2)
(m-3-3) edge [right hook->] (m-4-4)
(m-5-5) edge [right hook->] (m-6-6)
(m-1-3) edge [right hook->] (m-2-4)
(m-3-5) edge [right hook->] (m-4-6)
(m-3-1) edge [right hook->] (m-4-2)
(m-5-3) edge [right hook->] (m-6-4)

(m-3-5) edge [bend right,dashed,gray] node[above,near start] {$\cR$} (m-3-3)
(m-4-6) edge [bend right,dashed,gray] node[above,near end] {$\cR$} (m-4-4)
(m-5-5) edge [bend right,dashed,gray] node[above,near start] {$\cR$} (m-5-3)
(m-6-6) edge [bend right,dashed,gray] node[above,near end] {$\cR$} (m-6-4);
\end{tikzpicture}
\end{center}

The elements of $\TR$ are called {\bf trace polynomials} and the elements of $\cen$ are called {\bf pure trace polynomials}. Since every evaluation of $\free$ at a tuple of $n\times n$ matrices factors through $\TR$, it suffices to prove our Positivstellens\"atze in the ring $\TR$. The purpose of this reduction is of course not to merely state Theorem \ref{ta:free} in a more compact form. Our proofs crucially rely on algebraic properties of $\TR$ and their interaction with invariant and PI theory \cite{Pro,Row2}.

The contribution of this paper is twofold. We prove the Krivine-Stengle, Schm\"udgen and Putinar Positivstellens{\"a}tze for the trace ring of generic matrices in terms of cyclic quadratic modules and preorderings. We also prove Putinar's Positivstellensatz for the ring of generic matrices (without traces). The proofs intertwine techniques from invariant theory, real algebraic geometry, PI theory and functional analysis. Our second main result is a counterexample to the (universal) Procesi-Schacher conjecture.

\subsection{Main results and reader's guide}

After this introduction we recall known facts about polynomial identities, positive involutions, and the rings $\GM$ and $\TR$ in Section \ref{s:prelim}, where we also prove some preliminary results that are used in the sequel. 

Section \ref{s:ps3} deals with the question of Procesi and Schacher \cite{PS}, which is a noncommutative version of Hilbert's 17th problem for central simple $*$-algebras. We say that $a\in\USA$ is {\bf totally positive} if $a(X)\succeq0$ for every $X\in\mat^g$ where $a$ is defined. Then the universal Procesi-Schacher conjecture states that totally positive elements in $\USA$ are sums of hermitian squares in $\USA$. While this is true for $n=2$ \cite{PS,KU}, we show that the conjecture fails for $n=3$.

\begin{thmA}\label{ta:ps3}
There exist totally positive elements in $\usa_3$ that are not sums of hermitian squares in $\usa_3$.
\end{thmA}

The proof (see Theorem \ref{t:ps3}) relies on 
the central simple $*$-algebra
$\usa_3$ being split, i.e., $*$-isomorphic to $\opm_3(Z_3)$ with some orthogonal involution. After explicitly determining the involution using quadratic forms (Proposition \ref{p:idem2} of Tignol) and a transcendental basis of $Z_3$ (Lemma \ref{l:algindep}), we use Prestel's theory of semiorderings \cite{PD} to produce an  example of a totally positive element 
(a trace of a hermitian square)
in $\usa_3$ that is not a sum of hermitian squares (Proposition \ref{p:usa3}).\looseness=-1

Section \ref{s:ks} first introduces  cyclic quadratic modules and cyclic preorderings for the trace ring $\TR$, which are defined analogously as for $\free$ above. The main result in this section is the following version of the Krivine-Stengle Positivstellensatz for $\TR$.

\begin{thmA}\label{ta:ks}
Let $S\cup\{a\}\subset \STR$ be finite and $\cT$  the cyclic preordering generated by $S$.
\begin{enumerate}[label={\rm(\arabic*)}]
	\item $a|_{K_S}\succeq 0$ if and only if $at_1=t_1a=a^{2k}+t_2$ for some $t_1,t_2\in\cT$ and $k\in\N$.
	\item\label{it:ks2} $a|_{K_S}\succ 0$ if and only if $at_1=t_1a=1+t_2$ for some $t_1,t_2\in\cT$.
	\item $a|_{K_S}= 0$ if and only if $-a^{2k}\in\cT$ for some $k\in\N$.
\end{enumerate}
\end{thmA}

See Theorem \ref{t:posss} for the proof, which decomposes into three parts. First we show that the finite set of constraints $S\subset\STR$ can be replaced by a finite set $S'\subset\cen$ (Corollary \ref{c:scal}). To prove this central reduction we use the fact that the positive semidefiniteness of a matrix can be characterized by symmetric polynomials in its eigenvalues and apply  compactness of the real spectrum in the constructible topology \cite[Section 7.1]{BCR}. In the second step we apply results on central simple algebras with involution and techniques from PI theory to prove the following extension theorem.

\begin{thmA}\label{ta:hom}
Let $R\supseteq\R$ be a real closed field. Then an $\R$-algebra homomorphism $\phi:\cen\to R$ extends to an $\R$-algebra homomorphism $\R[\ulxi]\to R$ if and only if $\phi(\tr(hh^{\ti}))\ge0$ for all $h\in\TR$.
\end{thmA}

Since $\cen=\R[\ulxi]^{\ort}$, this statement resembles variants of the Procesi-Schwarz theorem \cite[Theorem 0.10]{PS1} (cf. \cite{CKS,Bro}). Nevertheless, it does not seem possible to deduce Theorem \ref{ta:hom} from these classical results; see Appendix \ref{a:ps} for a fuller discussion. Theorem \ref{ta:hom}, proved as Theorem \ref{t:hom} below, is essential for relating evaluations of pure trace polynomials with orderings on $\cen$ via Tarski's transfer principle (Proposition \ref{p:ord}). Finally, by combining the first two steps we obtain a reduction to the commutative situation, where we can apply an existing abstract version of the Krivine-Stengle Positivstellensatz \cite[Theorem 2.5.2]{Mar}.

Since trace polynomials are precisely $\ort$-concomitants in $\matpoly$, one might naively attempt to prove Theorem \ref{ta:ks} by simply applying the Reynolds operator for the $\ort$-action to analogous 
Positivstellens\"atze
for matrix polynomials \cite{Schm1,Cim}. However, the Reynolds operator is not multiplicative and it does not preserve squares of trace polynomials, so in this manner one obtains only weak and inadequate versions of Theorem \ref{ta:ks}.

In Section \ref{s:comp} we refine the strict positivity certificate \ref{it:ks2} of Theorem \ref{ta:ks} in the case of compact semialgebraic sets. We start by introducing {archimedean} cyclic quadratic modules, which encompass an algebraic notion of boundedness. Following the standard definition we say that a cyclic quadratic module $\cQ\subseteq\TR$ is {\bf archimedean} if for every $h\in\TR$ there exists $\rho\in\Q_{>0}$ such that $\rho-hh^{\ti}\in\cQ$. Then we prove Schm\"udgen's Positivstellensatz for trace polynomials.

\begin{thmA}\label{ta:sch}
Let $S\cup\{a\}\subset \STR$ be finite and $\cT$ be the cyclic preordering generated by $S$. If $K_S$ is compact and $a|_{K_S}\succ0$, then $a\in\cT$.
\end{thmA}

In the proof (see Theorem \ref{t:sch}) we apply  techniques similar to those in the proof of Theorem \ref{ta:ks}. That is, we replace $S$ by finitely many central constraints and apply Theorem \ref{ta:hom} to reduce to the commutative setting, where we use an abstract version of Schm\"udgen's Positivstellensatz \cite{Sch}.\looseness=-1

Finally, we have the following version
of Putinar's Positivstellensatz for 
$\TR$ and $\GM$, 
combining
 Theorems \ref{t:arch1} and \ref{t:arch}.

\begin{thmA}\label{ta:put}
\mbox\par{}
\begin{enumerate}[label={\rm(\alph*)}]
\item
Let $\cQ\subset\STR$ be an archimedean cyclic quadratic module and $a\in\STR$. If $a|_{K_{\cQ}}\succ0$, then $a\in \cQ$.
\item
Let $\cQ\subset\SGM$ be an archimedean quadratic module and $a\in\SGM$. If $a|_{K_{\cQ}}\succ0$, then $a\in \cQ$.
\end{enumerate}
\end{thmA}

Theorem \ref{ta:put} is proved in a more functional-analytic way. We start by assuming $a\notin \cQ$ and  find an extreme separation of $a$ and $\cQ$. Then we apply a Gelfand-Naimark-Segal construction towards finding a tuple of $n\times n$ matrices in $K_{\cQ}$ at which $a$ is not positive definite. For $\GM$ this is done using polynomial identities techniques, while for $\TR$ we use Theorem \ref{ta:hom}. As a consequence we have the following statement for noncommutative polynomials.

\begin{cora}\label{ca:freeput}
Let $\cQ\subset\Sym\R\mx$ be an archimedean quadratic module and $a\in\R\mx$. If $a|_{K_{\cQ}}\succ0$, then $a=q+f$ for some $q\in \cQ$ and $f\in\R\mx$ satisfying $f|_{\mat^g}=0$.
\end{cora}

\begin{proof}
If $\pi:\R\mx\to\GM$ is the canonical $*$-homomorphism, then $\pi(\cQ)$ is an archimedean module in $\GM$ and hence $\pi(a)\in\pi(\cQ)$ by Theorem \ref{ta:put}(b). Corollary \ref{ca:freeput} now follows because the kernel of $\pi$ consists precisely of the polynomial identities for $n\times n$ matrices.
\end{proof}

The paper concludes with Section \ref{s:exa} containing examples and counterexamples. In Appendix \ref{a:rey} where we present algebraic constructions of the Reynolds operator for the action of 
$\ort$ on polynomials and matrix polynomials as alternatives to the integration over the orthogonal group, which is of interest in mathematical physics \cite{CS}. Appendix \ref{a:ps} explains  why the Procesi-Schwarz theorem cannot be used to obtain the extension theorem \ref{ta:hom}.

\subsection*{Acknowledgments}
The authors thank Jean-Pierre Tignol for sharing his expertise and generously allowing us to include his ideas that led to the counterexample for the universal Procesi-Schacher conjecture, and James Pascoe for his thoughtful suggestions. We also acknowledge fruitful Oberwolfach discussions with Cordian Riener and Markus Schweighofer.

%%%%%%%%%%%%%%%%%%%%%%%%%%%%%%%%%
%%%%%%%%%%%%%%%%%%%%%%%%%%%%%%%%%

\section{Preliminaries}\label{s:prelim}

In this section we collect some background material and preliminary results needed in the sequel.

\subsection{Polynomial and trace \texorpdfstring{$*$}{*}-identities}

Throughout the paper let $F$ be a field of characteristic $0$. Let $\ulx=\{x_1,\dots,x_g\}$ and $\ulx^*=\{x_1^*,\dots,x_g^*\}$ be freely noncommuting variables, and let $\mx$ be the free monoid generated by $x_j,x_j^*$. The free algebra $\px$ is then endowed with the unique involution of the first kind determined by $x_j\mapsto x_j^*$. If $\cA$ is an $F$-algebra with involution $\tau$ and $f=f(x_1,\dots,x_g,x_1^*,\dots,x_g^*)\in\px$ is such that
$$f(a_1,\dots,a_g,a_1^\tau,\dots,a_g^\tau)=0$$
for all $a_j\in\cA$, then $f$ is a {\bf polynomial $*$-identity} of $(\cA,\tau)$.

Let $\sim$ be the equivalence relation on $\mx$ generated by
$$w_1w_2\sim w_2w_1,\qquad w_1\sim w_1^*$$
for $w_1,w_2\in\mx$. Let $\Tr(w)$ be the equivalence class for $w\in\mx$. Then we define the {\bf free trace ring with involution} $\free=T\otimes_F \px$, where $T$ is the free commutative $F$-algebra generated by $\Tr(w)$ for $w\in\mx\!/\!\!\sim$. Note that $\Tr(1)\in \free$ is one of the generators of $T$ and not a real scalar. If $\cA$ is an $F$-algebra, then an $F$-linear map $\chi:\cA\to F$ satisfying $\chi(ab)=\chi(ba)$ for $a,b\in\cA$ is called an {\bf $F$-trace} on $\cA$. If
$$f=\sum_i \alpha_i\Tr(w_{i1})\cdots\Tr(w_{i\ell_i})w_{i0},\qquad \alpha_i\in F, \ w_{ij}\in\mx$$
satisfies
$$\sum_i \alpha_i\chi(w_{i1}(a))\cdots\chi(w_{i\ell_i}(a))w_{i0}(a)=0$$
for every tuple $a$ of elements in $\cA$, then $f$  is a {\bf trace $*$-identity} of $(\cA,\tau,\chi)$.

\subsubsection{A particular trace \texorpdfstring{$*$}{*}-identity}

For $n\in\N$ let $\ti$ denote the transpose involution on $\opm_n(F)$ and let $\si$ denote the symplectic involution on $\opm_{2n}(F)$:
$$\begin{pmatrix} a & b \\ c & d\end{pmatrix}^{\si}=\begin{pmatrix} d^{\ti} & -b^{\ti} \\ -c^{\ti} & a^{\ti}\end{pmatrix}.$$
Let $\tr:\opm_n(F)\to F$ be the usual trace. Finally, for $X\in\opm_n(F)$ let $X^{\oplus d}\in\opm_{dn}(F)$ denote the block-diagonal matrix with $d$ diagonal blocks all equal to $X$.

Fix $m\in\N$. For a $\ti$-antisymmetric $A\in\opm_{2m}(F)$ let $\pf(A)\in F$ be its Pfaffian, $\pf(A)^2=\det(A)$. Suppose that $A_1,A_2\in\opm_{2m}(F)$ are $\ti$-antisymmetric and $A_2$ is invertible. Now consider
$$f=\pf(A_2)\pf(tA_2^{-1}-A_1)\in F[t].$$
Then $f^2$ is the characteristic polynomial of $A_1A_2$, so $\pm f$ is monic of degree $m$ and the coefficients of $f$ are polynomials in the entries of $A_1,A_2$ by Gauss' lemma. Also, as in the proof of the Cayley-Hamilton theorem we see that $f(A_1A_2)=0$.

Hence for every $\ti$-antisymmetric $A_1,A_2\in\opm_{2m}(F)$ there exists $f=t^m+\sum_k (-1)^k c_k t^{m-k}\in F[t]$ such that $f(A_1A_2)=0$. If $A_1A_2$ has distinct eigenvalues $\lambda_1,\dots,\lambda_m$, then their blocks in the Jordan decomposition of $A_1A_2$ have multiplicity $2$ and
$$2\left(\sum_{j=1}^m\lambda_j^i\right)=\tr\left((A_1A_2)^i\right)$$
for $i\in\N$. Now Newton's identities imply
$$kc_k=\sum_{i=1}^k \frac12 (-1)^{i-1}\tr\left((A_1A_2)^k\right) c_{k-i}$$
for $1\le k\le m$ and $c_0=1$.

Now define $f_m\in\free$ as
$$f_m=\sum_{k=0}^m (-1)^k f'_k\cdot(x_1x_2)^{m-k}$$
with $f'_0=1$ and
\begin{equation}\label{e:rec}
f'_k=\sum_{i=1}^k \frac{1}{2k} (-1)^{i-1}\tr\left((x_1x_2)^k\right) f'_{k-i}
\end{equation}
for $1\le k\le m$. The following lemmas will be important for distinguishing between different types of involutions of the first kind in the sequel.

\begin{lem}\label{l:skew0}
For every $m\in\N$, $f_m(x_1-x_1^*,x_2-x_2^*)$ is a $*$-trace identity of $(\opm_{2m}(F),\ti,\tr)$.
\end{lem}

\begin{proof}
Observe that the set of pairs of $\ti$-antisymmetric $A_1,A_2\in\opm_{2m}(F)$, such that $A_1A_2$ has $m$ distinct eigenvalues, is Zariski dense in the set of all pairs of $\ti$-antisymmetric $A_1,A_2\in\opm_{2m}(F)$. Hence the conclusion follows by the construction of $f_m$.
\end{proof}

\begin{lem}\label{l:skew}
For every $n,m\in\N$ and $d\in \N\setminus 2\N$ there exist $\si$-antisymmetric $A_1,A_2\in\opm_{2n}(F)$ such that
$$f_m\left(A_1^{\oplus d}A_2^{\oplus d}\right)\neq0.$$
\end{lem}

\begin{proof}
Every $\ti$-symmetric matrix $S\in\opm_{2n}(F)$ can be written as $S=(-SJ)J$, where $J=\left(\begin{smallmatrix}0 & 1 \\ -1 & 0\end{smallmatrix}\right)$ and $-SJ,J$ are $\si$-antisymmetric matrices. Hence it suffices to prove that $f_m(S^{\oplus d})\neq0$ holds for
$$S=\diag(\overbrace{1,\dots,1}^{2n-1},0)\in\opm_{2n}(F).$$
Since $\tr((S^{\oplus d})^k)=d(2n-1)$ is odd, we can use \eqref{e:rec} and induction on $k$ to show that
$$k!f'_k(S^{\oplus d})\in
\left\{\frac{\ell}{2^k}\colon \ell\in\Z\right\}\setminus\left\{\frac{\ell}{2^{k-1}}\colon \ell\in\Z\right\}$$
for $1\le k\le m$. In particular we have $f'_m(S^{\oplus d})\neq0$ and thus $f_m(S^{\oplus d})\neq0$.
\end{proof}

For $n\in\N$ and $m\in2\N$ let $\JJ(n,\ti)$ denote the set of polynomial $*$-identities of $(\opm_n(F),\ti)$ and let $\JJ(m,\si)$ denote the set of polynomial $*$-identities of $(\opm_m(F),\si)$. By \cite[Corollary 2.5.12 and Remark 2.5.13]{Row2} we have $\JJ(m,\si)\subseteq \JJ(n,\ti)$ if and only if $2n\le m$.

\begin{prop}\label{p:ts}
Let $n\in\N$ and $m\in2\N$. Then $\JJ(n,\ti)\subseteq \JJ(m,\si)$ if and only if $2m\le n$.
\end{prop}

\begin{proof}
$(\Rightarrow)$ Let
$$c_m=\sum_{\pi\in \Sym_m} \sgn{\pi}
x_{\pi(1)}x_{m+1}x_{\pi(2)}x_{m+2}\cdots x_{2m-1}x_{\pi(m)}$$
be the $m$th Capelli polynomial \cite[Section 1.2]{Row2}. If $\cA$ is a central simple $F$-algebra and $a_1,\dots,a_m\in\cA$, then $\{a_1,\dots,a_m\}$ is linearly dependent over $F$ if and only if
$$c_g(a_1,\dots,a_g,b_1,\dots,b_{m-1})=0 \qquad \forall b_i\in\cA$$
by \cite[Theorem 1.4.34]{Row2}.

Now assume $n<2m$. If $A_1,A_2\in\opm_n(F)$ are $\ti$-antisymmetric, then the set
$$\left\{A_1A_2,\dots,(A_1A_2)^{\lfloor n/2\rfloor+1}\right\}$$
is linearly dependent. Indeed, for an even $n$ this holds directly by Lemma \ref{l:skew0}, while for an odd $n$ we use the fact that $A_1A_2$ is singular and then apply Lemma \ref{l:skew0} for $n+1$. On the other hand, since every $\ti$-symmetric matrix in $\opm_m(F)$ is a product of two $\si$-antisymmetric matrices, there exist $\si$-antisymmetric $A_1,A_2\in\opm_m(F)$ such that $\{1,\dots,(A_1A_2)^{m-1}\}$ is linearly independent. Since $\lfloor \frac{n}{2}\rfloor+1\le m$,
\begin{align*}
c_m\Big((x_1-x_1^*)(x_2-x_2^*),\dots,((x_1-x_1^*)(x_2-x_2^*))^m,x_3,\dots,x_{m+1}\Big)
\end{align*}
is a $*$-identity of $\opm_n(F)$ endowed with $\ti$ but is not a $*$-identity of $\opm_m(F)$ endowed with $\si$.

$(\Leftarrow)$ By \cite[Corollary 2.3.32]{Row2} we can assume that $F$ is algebraically closed; let $i\in F$ be such that $i^2=-1$. Since $m\in2\N$, $(\opm_m(F),\si)$ $*$-embeds into $(\opm_{2m}(F),\ti)$ via
\[(\opm_m(F),\si)\hookrightarrow (\opm_{2m}(F),\ti),\qquad
\begin{pmatrix}a & b \\ c & d\end{pmatrix}\mapsto
\frac12\begin{pmatrix}
a+d & i(a-d) & c-b & i(b+c)\\
i(d-a) & a+d & i(b+c) & b-c\\
b-c & -i(b+c) & a+d & i(a-d)\\
-i(b+c) & c-b & i(d-a) & a+d 
\end{pmatrix}.
\]
\end{proof}

\begin{rem}\label{r:ts}
The same reasoning as in the proof of Proposition \ref{p:ts} also implies that elements of $\JJ(n,\ti)$ are  polynomial *-identities of $\opm_m(F)$ with an involution of the second kind if and only if $2m\le n$. Recall that an involution on $\opm_m(F)$ is of the second kind if 
it induces an automorphism of order two on $F$.
\end{rem}

\subsection{Generic matrices and the trace ring}\label{ss:gentr}

For $g,n\in\N$ let
$$\ulxi=\left\{\xi_{j\imath\jmath}\colon 1\le j\le g, 1\le \imath,\jmath\le n\right\}$$
be a set of commuting indeterminates. We recall the terminology from Section \ref{s:intro}. Let
$$\Xi_j=(\xi_{j\imath\jmath})_{\imath\jmath}\in\matpoly$$
be $n\times n$ generic matrices and let $\GM\subset\matpoly$ be the $\R$-algebra generated by $\Xi_j$ and their transposes $\Xi_j^{\ti}$. Furthermore, let $\TR\subset\matpoly$ be the $\R$-algebra generated by $\GM$ and traces of elements in $\GM$. This algebra is called the {\bf trace ring} of $n\times n$ generic matrices (see e.g.~\cite[Section 7]{Pro}) and inherits the transpose involution $\ti$ and trace $\tr$ from $\matpoly$. Let $\ceng\subset\cen\subset\R[\ulxi]$ be the centers of $\GM$ and $\TR$, respectively. The elements of $\TR$ are called {\bf trace polynomials} and the elements of $\cen$ are called {\bf pure trace polynomials}.

There is another, more invariant-theoretic description of the trace ring. Define the following action of the orthogonal group $\ort$ on $\mat^g$:
\begin{equation}\label{e:act}
(X_1,\dots,X_g)^u := (uX_1u^{\ti},\dots,uX_gu^{\ti}),\qquad X_j\in\mat,\ u\in\ort
\end{equation}
and consider $\R[\ulxi]$ as the coordinate ring of $\mat^g$. By \cite[Theorems 7.1 and 7.2]{Pro}, $\cen$ is the ring of $\ort$-invariants in $\R[\ulxi]$ and $\TR$ is the ring of $\ort$-concomitants in $\matpoly$, i.e., elements $f\in\matpoly$ satisfying
$$f(X^u)=uf(X)u^{\ti}$$
for all $X\in\mat^g$ and $u\in\ort$.

We list a few important properties of $\GM$ and $\TR$ that will be used frequently in the sequel.

\begin{enumerate}[label={\rm(\alph*)}]
\item Let $\JJ(n,\ti,\tr)\subset\free$ denote the set of trace $*$-identities of $(\mat,\ti,\tr)$. Then $\GM\cong \R\mx/\JJ(n,\ti)$ by \cite[Remark 3.2.31]{Row2} and $\TR\cong\free/\JJ(n,\ti,\tr)$  by \cite[Theorem 8.4]{Pro}.

\item By \cite[Theorem 20.1]{Pro}, the ring of central quotients of $\GM$ is a central simple algebra of degree $n$,  which is also the ring of rational $\ort$-concomitants in $\matrat$. It is called the {\bf universal central simple algebra with orthogonal involution} of degree $n$. We denote it by $\USA$ and its center by $Z_n$. Note that $\USA$ is also the ring of central quotients of $\TR$.

\item By \cite[Theorem 7.3]{Pro}, $\cen$ is a finitely generated $\R$-algebra and $\TR$ is finitely spanned over $\cen$. In particular, $\cen$ and $\TR$ are Noetherian rings.
\end{enumerate}

\subsubsection{Reynolds operator}

This subsection is to recall some basic properties of the Reynolds operator \cite[Subsection 2.2.1]{DK}. Let $G$ be an algebraic group and $X$ an affine $G$-variety. The {\bf Reynolds operator} $\cR:F[X]\to F[X]^G$ is a linear map with the properties:
\begin{enumerate}[label={\rm(\arabic*)}]
\item $\cR(f)=f$ for $f\in F[X]^G$,
\item $\cR$ is a $G$-module homomorphism; i.e., $\cR(f^u)=\cR(f)$ for $f\in F[X],u\in G$.
\end{enumerate}
The Reynolds operator is hence a $G$-invariant projection onto the space of the invariants. The Reynolds operator exists if $G$ is linearly reductive and is then unique (see e.g.~\cite[Theorem 2.2.5]{DK}).

Let $M,N$ be $G$-modules and $f:M\to N$ a $G$-module homomorphism. 
Denote by $M^G,N^G$ the modules of invariants of $M$, $N$, resp., the corresponding Reynolds operators by $\cR_M,\cR_N$, resp., and $f^G$  a $G$-module homomorphism $f$ restricted to $M^G$. 
Then $\cR_Nf=f^G\cR_M$. This easily follows by the uniqueness of the Reynolds operator.  
The Reynolds operator is thus functorial.

In our case $\ort$ acts on $\mat^g$ by simultaneous conjugation as in \eqref{e:act}. Since $\matpoly$ can be identified with polynomial maps $\mat^g\to\mat$, we have the Reynolds operator $\rey:\matpoly\to\TR$ with respect to the action \eqref{e:act}. Since $\ort$ is a compact Lie group, $\rey$ can be given by the averaging integral formula (with respect to the normalized left Haar measure $\mu$ on $\ort$)
\begin{equation}\label{e:rey2}\rey(f)=\int_{\ort} f^u \,d\mu(u).\end{equation}
Consequently $\rey$ is a trace-intertwining $\TR$-module homomorphism, i.e.,
\begin{equation}\label{e:rey3}
\rey(hf)=h\rey(f),\qquad \rey(fh)=\rey(f)h,\qquad \tr(\rey(f))=\rey(\tr(f))
\end{equation}
for all $h\in\TR$ and $f\in\matpoly$. In Appendix \ref{a:rey} we present algebraic ways of computing $\rey$.

\subsection{Positive involutions and totally positive elements}\label{ss:posinv}

Let $\cA$ be a central simple algebra with involution $\tau$ and $*$-center $F$ (that is, $F$ is the subfield of $*$-invariant elements in the center of $\cA$). The $F$-space of $\tau$-symmetric elements in $\cA$ is denoted $\Sym\cA$. Following the terminology of \cite{PS} and \cite{KU}, an ordering $\ge$ of $F$ is a {\bf $*$-ordering} if $\tr_{\cA}(aa^\tau)\ge0$ for every $a\in\cA$. In this case we also say that $\tau$ is {\bf positive} with respect to such an ordering. An element $a\in\Sym\cA$ is {\bf positive} in a given $*$-ordering if the hermitian form $x\mapsto \tr(x^\tau ax)$ on $\cA$ is positive semidefinite. Finally, $a\in\Sym\cA$ is {\bf totally positive} if it is positive with respect to every $*$-ordering.

Let $\alpha_1,\dots,\alpha_n\in F$ be the elements appearing in a diagonalization of the form $x\mapsto \tr(xx^\tau)$ on $\cA$. By \cite[Theorem 5.4]{PS}, a symmetric $s\in\cA$ is totally positive if and only if it has a weighted sum of hermitian squares representation
\begin{equation}\label{e:p}
s=\sum_{I\in\{0,1\}^n}\alpha^I\sum_ih_{I,i}h_{I,i}^\tau,
\end{equation}
where $\alpha^I=\alpha_1^{I_1}\cdots\alpha_n^{I_n}$ and $h_{I,i}\in\cA$.

Let $\TPc\subset \cen$ be the {\bf preordering} generated by $\tr(hh^{\ti})$ for $h\in\TR$, i.e., the set of all sums of products of $\tr(hh^{\ti})$ (note that $c^2=\tr((\frac{c}{\sqrt{n}})^2)\in\TPc$ for every $c\in\cen$, so $\TPc$ is really a preordering). Further, let
$$\TP=\left\{\sum_i\omega_ih_ih_i^{\ti}\colon \omega_i\in\TPc,h_i\in\TR\right\}.$$
Note that $\TPc=\tr(\TP)$.

\begin{lem}\label{l:rey}
Let $f\in\Sym\matpoly$. If $f(X)\succeq0$ for all $X\in\mat^g$, then $\rey(f)=c^{-2}q$ for some $q\in\TP$ and $c\in \cen\setminus\{0\}$.
\end{lem}

\begin{proof}
By the integral formula \eqref{e:rey2} it is clear that $\rey(f)(X)\succeq0$ for all $X\in\mat^g$. Hence $\rey(f)$ is a totally positive element in $\USA$ by \cite[Lemma 5.3]{KU}, so
$$\rey(f)=\sum_{I\in\{0,1\}^{n^2}}\alpha^I\sum_ih_{I,i}h_{I,i}^{\ti}$$
for some $h_{I,i}\in\USA$ and a diagonalization $\langle\alpha_1,\dots,\alpha_{n^2}\rangle$ over $Z_n$ of the form $x\mapsto \tr(xx^{\ti})$ on $\USA$. Hence $\alpha_k=\tr(\tilde{h}_k\tilde{h}_k^{\ti})$ for some $\tilde{h}_k\in\USA$. Since $\USA$ is the ring of central quotients of $\TR$, there exist $q\in\TP$ and $c\in \cen$ such that $\rey(f)=c^{-2}q$.
\end{proof}

As demonstrated in Example \ref{exa:det}, the denominator in Lemma \ref{l:rey} is in general indispensable even if $f$ is a hermitian square or $f\in\R[\ulxi]$. For more information about images of squares under Reynolds operators for reductive groups acting on real affine varieties see \cite{CKS}.

\begin{rem}
In particular, the linear operator $\rey$ does not map squares in $\R[\ulxi]$ into $\TPc$ or hermitian squares in $\matpoly$ into $\TP$. Hence our Positivstellens\"atze in the sequel cannot simply be deduced from their commutative or matrix counterparts
by averaging with $\rey$. Furthermore, even if one were content with using totally positive polynomials (which by Lemma \ref{l:rey} are of the form $c^{-2}q$ for $q\in\TP$ and $c\in\cen\setminus\{0\}$) instead of $\TP$, one could still  not derive our results since $\rey$ is not multiplicative.
\end{rem}

%%%%%%%%%%%%%%%%%%%%%%%%%%%%%%%%%
%%%%%%%%%%%%%%%%%%%%%%%%%%%%%%%%%

\section{Counterexample to the \texorpdfstring{$3\times3$}{3x3} universal Procesi-Schacher conjecture}\label{s:ps3}

By \cite[Corollary 5.5]{PS} every totally positive element in $\usa_2$ is a sum of hermitian squares, i.e., of the form $c^{-2}\sum_ih_ih_i^{\ti}$ for $c\in C_2$ and $h_i\in\gm_2$. Indeed, by \eqref{e:p} it suffices to show that $\tr(aa^{\ti})$ is a sum of hermitian squares. Since $\usa_2$ is a division ring, we have
$$\tr(aa^{\ti})=a^{\ti}a+(\det(a)a^{-1})(\det(a)a^{-1})^{\ti}$$
for every $a\in\usa_2\setminus\{0\}$ by the Cayley-Hamilton theorem. In their 1976 paper \cite{PS}, Procesi and Schacher asked if the same holds true for $n>2$:

\begin{conj}[The universal Procesi-Schacher conjecture]\label{co:ps}
Let $n\ge 2$. Then every totally positive element in $\USA$ is a sum of hermitian squares.
\end{conj}

By \eqref{e:p}, Conjecture \ref{co:ps} is equivalent to the following: every trace of a hermitian square in $\USA$ is a sum of hermitian squares in $\USA$. In this section we show that Conjecture \ref{co:ps} fails for $n=3$:

\begin{thm}\label{t:ps3}
There exist totally positive elements in $\usa_3$ that are not sums of hermitian squares in $\usa_3$.
\end{thm}

In the first step of the proof we identify the split central simple algebra $\usa_3$ as a matrix algebra $\opm_3(F)$ for a rational function field $F$, endowed with an involution of the orthogonal type. For the constructive proof of Theorem \ref{t:ps3} we then use Prestel's theory of semiorderings \cite{PD}.

We recall some terminology of quadratic forms from \cite{KMRT}. Let $F$ be a field and $V$ an $n$-dimensional vector space. Quadratic forms $q$ and $q'$ are equivalent if there exists $\theta\in\GL_FV$ such that $q'=q\circ \theta$. Quadratic forms $q$ and $q'$ are similar if $\alpha q$ and $q'$ are equivalent for some $\alpha\in F\setminus\{0\}$. Every quadratic form is equivalent to a diagonal quadratic form, which is denoted $\langle \alpha_1,\dots,\alpha_n\rangle$ for $\alpha_i\in F$.

First we fix $g=1$ and write $\Xi=\Xi_1$. Since $\usa_3$ is an odd degree central simple algebra with involution of the first kind, $\usa_3$ is split by \cite[Corollary 2.8]{KMRT}. Let us fix a $*$-representation $\usa_3=\End_{Z_3}V$, where $V$ is a 3-dimensional vector space over $Z_3$, the center of $\usa_3$. By \cite[Proposition 2.1]{KMRT}, there exists a symmetric bilinear form $b:V\times V\to Z_3$ such that
$$b(xu,v)=b(u,x^{\ti}v)$$
for all $u,v\in V$ and $x\in\End_{Z_3}V$, where ${\ti}$ denotes the involution on $\End_{Z_3}V$ originating from $\usa_3$. Let $q:V\to Z_3$ given by $q(u)=b(u,u)$ be the associated quadratic form.

\begin{lem}\label{l:idem1}
Let $a\in\End_{Z_3}V$ be $\ti$-antisymmetric with $\tr(a^2)\neq0$. Define $e=1-2\tr(a^2)^{-1}a^2$. Then $e$ is a symmetric idempotent of rank 1 such that $V=\im e\perp \ker e$. Moreover, $\im a=\ker e$ and $\ker a=\im e$, and the determinant of the restriction of $q$ to $\ker e$ is $-\frac{1}{2}\tr(a^2)$.
\end{lem}

\begin{proof}
Since $a$ and $a^{\ti}=-a$ have the same trace and determinant, we have $\tr(a)=\det(a)=0$, so by the Cayley-Hamilton theorem it follows that
\begin{equation}\label{e:anti3}
a^3-\frac{1}{2}\tr(a^2)a=0.
\end{equation}
Hence $a^4=\frac{1}{2}\tr(a^2)a^2$ and it is straightforward to check that $e$ is a symmetric idempotent. It has rank 1 because $\tr(e)=1$, and the decomposition $V=\im e\oplus\ker e$ is orthogonal because $e$ is symmetric. The equation \eqref{e:anti3} also yields $ea=ae=0$, hence $\im a\subseteq\ker e$ and $\im e\subseteq \ker a$. Since the rank of every antisymmetric matrix is even, we have $\im a=\ker e$ and $\im e=\ker a$.

To prove the last statement, observe that the restriction of $a$ to $\ker e$ is an antisymmetric operator with determinant $-\frac{1}{2}\tr(a^2)$, and the determinant of the restriction of $q$ to $\ker e$ is the square class of the determinant of any nonzero antisymmetric operator; see \cite[Proposition 7.3]{KMRT}.
\end{proof}

\begin{prop}\label{p:idem2}
For $i=1,2$ let $a_i\in\End_{Z_3}V$ be $\ti$-antisymmetric with $\tr(a_i^2)\neq0$, and let $e_i=1-2\tr(a_i^2)^{-1}a_i^2$. If $e_1e_2=e_2e_1=0$, then $q$ is similar to $\langle 1,-\frac{1}{2}\tr(a_1^2),-\frac{1}{2}\tr(a_2^2)\rangle$.
\end{prop}

\begin{proof}
Let $e_3=1-e_1-e_2$ and $V_i=\im e_i$ for $i=1,2,3$; we have $\dim V_i=1$ for each $i$. Moreover, if $u\in V_i$ and $v\in V_j$ for $i\neq j$, then
$$b(u,v)=b(e_iu,e_jv)=b(u,e_ie_jv)=0.$$
Therefore $V=V_1\perp V_2\perp V_3$. We have $\ker e_1=V_2\perp V_3$ and $\ker e_2=V_1\perp V_3$, and Lemma \ref{l:idem1} shows that the determinant of the restriction of $q$ to $\ker e_i$ is $-\frac{1}{2}\tr(a_i^2)$ for $i=1,2$. If $\alpha\in Z_3\setminus\{0\}$ is such that the restriction of $q$ to $V_3$ is equivalent to $\langle\alpha\rangle$, then it follows that the restriction of $q$ to $V_i$ is equivalent to $\langle-\frac{\alpha}{2}\tr(a_i^2)\rangle$ for $i=1,2$. Hence $q$ is equivalent to $\langle \alpha,-\frac{\alpha}{2}\tr(a_1^2),-\frac{\alpha}{2}\tr(a_2^2)\rangle$.
\end{proof}

Now let
\begin{alignat*}{2}
a_1 & =\Xi-\Xi^{\ti},\qquad & a_2 & =e_1\Xi(1-e_1)-(1-e_1)\Xi^{\ti} e_1,\\
e_1 & =1-2\tr(a_1^2)^{-1}a_1^2, \qquad & e_2 & =1-2\tr(a_2^2)^{-1}a_2^2, \\
\beta_1 &=-\frac12 \tr(a_1^2),\qquad & \beta_2 &=-\frac12 \tr(a_2^2).
\end{alignat*}
Since $a_1,a_2$ are nonzero, it follows from $a_i^3-\frac{1}{2}\tr(a_i^2)a_i=0$ that $\beta_i\neq0$ for $i=1,2$. It is also easy to check that $e_1e_2=e_2e_1=0$, so the conclusion of Proposition \ref{p:idem2} holds. Hence we can choose a basis of $V$ in such a way that
\begin{equation}\label{e:ad}
x^{\ti}=\diag(1,\beta_1,\beta_2)^{-1}x^\tau\diag(1,\beta_1,\beta_2)
\end{equation}
for all $x\in\End_{Z_3}V$, where $\tau$ is the transpose in $\opm_3(Z_3)=\End_{Z_3}V$ with respect to the chosen basis of $V$.

The field $Z_3$ is rational over $\R$ by \cite[Theorem 1.2]{Sal} and of transcendental degree $6$ by \cite[Theorem 1.11]{BS}. Inspired by \cite[Section 3]{For0} we present an explicit transcendental basis for $Z_3$ over $\R$. Denote
$$s=\frac12(\Xi+\Xi^{\ti}),\qquad a= \frac12(\Xi-\Xi^{\ti}),\qquad s_0=s-\frac13 \tr(s)$$
and
\begin{equation}\label{e:alphas}
\begin{aligned}
\alpha_1 &=\tr(s), \qquad & \alpha_4 &=\frac{\tr(a^2)^2\tr(s_0^2)-6\tr(s_0a^2)^2}{\tr(a^2)^2\tr(s_0^2)-4\tr(a^2)\tr(s_0^2a^2)-2\tr(s_0a^2)^2}, \\
\alpha_2 &=\tr(a^2), \qquad & \alpha_5 &=\frac{\tr(a^2)^3\tr(s_0^3)+6\tr(s_0a^2)^3}{\tr(a^2)^2\tr(s_0^2)-6\tr(s_0a^2)^2}, \\
\alpha_3 &=\tr(s_0a^2), \qquad & \alpha_6 
&=\frac{\tr(as_0a^2s_0^2)}{\tr(a^2)^2\tr(s_0^2)-6\tr(s_0a^2)^2}.
\end{aligned}
\end{equation}

\begin{lem}\label{l:algindep}
The elements $\alpha_1,\dots,\alpha_6$ are algebraically independent over $\R$, $Z_3=\R(\alpha_1,\dots,\alpha_6)$, and
\begin{equation}\label{e:expr}
\beta_1=-\frac12\alpha_2,\qquad
\beta_2=\frac{288\alpha_2^3\alpha_4^2\alpha_6^2-(3\alpha_3\alpha_4+2\alpha_4\alpha_5+9\alpha_3)^2}{9\alpha_2^2(\alpha_4+1)}.
\end{equation}
\end{lem}

\begin{proof}
Using a computer algebra system one can verify that the determinant of the Jacobian matrix $\operatorname{J}_{\alpha_1,\dots,\alpha_6}$ is nonzero, so by the Jacobian criterion $\alpha_1,\dots,\alpha_6$ are algebraically independent over $\R$. Likewise, \eqref{e:expr} is checked by a computer algebra system.

A minimal set of generators of pure trace polynomials in two $3\times 3$ generic matrices without involution is given in \cite[Section 1]{ADS} or in the proof of \cite[Proposition 7]{LV}. Replacing the first generic matrix by $s$ and the second generic matrix by $a$ we obtain the following generators of $T_3$:
\begin{equation}\label{e:gen}
\tr(s),\quad \tr(s_0^2),\quad \tr(s_0^3),\quad \tr(a^2),\quad \tr(s_0a^2),\quad \tr(s_0^2a^2),\quad \tr(as_0a^2s_0^2).
\end{equation}
From \eqref{e:alphas} we can directly see that \eqref{e:gen} are rational functions in $\alpha_1,\dots,\alpha_6,\tr(s_0)^2$; for example,
$$\tr(s_0^3)=\frac{\alpha_5(\alpha_2^2\tr(s_0^2)-6\alpha_3^2)-6\alpha_3^3}{\alpha_2^3}.$$
Then we use a computer algebra system to verify that
$$\tr(s_0^2)=\frac{2\alpha_4}{\alpha_4+1}\beta_2 + \frac{6\alpha_3^2}{\alpha_2^2}$$
is a rational function in $\alpha_1,\dots,\alpha_6$ and hence $Z_3=\R(\alpha_1,\dots,\alpha_6)$.
\end{proof}

\begin{prop}\label{p:usa3}
$\beta_1\beta_2\in Z_3$ is totally positive in $\usa_3$ but is not a sum of hermitian squares in $\usa_3$.
\end{prop}

\begin{proof}
Since $\beta_1\beta_2=\tr(hh^{\ti})$ for
$$h=\begin{pmatrix}0 & 0 & 0 \\ 0 & 0 & \beta_2 \\ 0 & 0 & 0 \end{pmatrix},$$
$\beta_1\beta_2$ is totally positive in $\usa_3$. Now suppose $\beta_1\beta_2=\sum_ir_ir_i^{\ti}$ for $r_i\in\usa_3$. If $r_i=(\rho_{i\imath\jmath})_{\imath\jmath}$, then the $(1,1)$-entry of $\sum_ir_ir_i^{\ti}$ equals
$$\sum_i\left(\rho_{i11}^2+\beta_1^{-1}\rho_{i12}^2+\beta_2^{-1}\rho_{i13}^2\right)$$
and therefore
\begin{equation}\label{e:repr1}
1=\beta_1\beta_2\sum_i\left(\frac{\rho_{i11}}{\beta_1\beta_2}\right)^2
+\beta_2\sum_i\left(\frac{\rho_{i12}}{\beta_1\beta_2}\right)^2
+\beta_1\sum_i\left(\frac{\rho_{i13}}{\beta_1\beta_2}\right)^2.
\end{equation}
By \cite[Exercise 5.5.3 and Lemma 5.1.8]{PD} there exists a semiordering $Q\subset\R(\alpha_1,\dots,\alpha_6)$ satisfying $\alpha_2,-\alpha_4,-\alpha_2\alpha_4\in Q$ and $p\in Q\cap \R[\alpha_1,\dots,\alpha_6]$ if and only if the term of the highest degree in $p$ belongs to $Q$. These assumptions on $Q$ yield $-\beta_1,-\beta_2,-\beta_1\beta_2\in Q$, so \eqref{e:repr1} implies $-1\in Q$, a contradiction.
\end{proof}

\begin{proof}[Proof of Theorem \ref{t:ps3}]
To be more precise we write $\usa_{3,g}$ for $\usa_3$ generated by $g$ generic matrices $\Xi_j$. Proposition \ref{p:usa3} proves Theorem \ref{t:ps3} for $g=1$. Now let $g\in\N$ be arbitrary; note that $\usa_{3,1}$ naturally $*$-embeds into $\usa_{3,g}$. Let $s\in\usa_{3,1}$ be a totally positive element that is not a sum of hermitian squares in $\usa_{3,1}$. Suppose that $s$ is a sum of hermitian squares in $\usa_{3,g}$, i.e.,
$$s=c^{-2}\sum_jh_ih_i^{\ti}$$
for some $c,h_i\in\gm_{3,g}$ with $c$ central. Since the sets of polynomial $*$-identities of $\gm_{3,1}$ and $\gm_{3,g}$ coincide, it is easy to see that there exists a $*$-homomorphism $\phi:\gm_{3,g}\to\gm_{3,1}$ satisfying $\phi(\Xi_1)=\Xi_1$ and $\phi(c)\neq0$. Then
$$s=\phi(c)^{-2}\sum_j\phi(h_i)\phi(h_i)^{\ti}$$
is a sum of hermitian squares in $\usa_{3,1}$, a contradiction.
\end{proof}

We do not know if Conjecture \ref{co:ps} holds for $n=4$, where $\usa_4$ is a division biquaternion algebra by \cite[Theorem 20.1]{Pro} and hence does not split.
In \cite{AU} the authors use signatures
of hermitian forms to distinguish between
sums of hermitian squares and general
totally positive elements.

%%%%%%%%%%%%%%%%%%%%%%%%%%%%%%%%%
%%%%%%%%%%%%%%%%%%%%%%%%%%%%%%%%%

\section{The Krivine-Stengle Positivstellensatz for trace polynomials}\label{s:ks}

In this section we prove the Krivine-Stengle Positivstellensatz representing trace polynomials positive on semialgebraic sets in terms of weighted sums of hermitian squares with denominators.

\subsection{Cyclic quadratic modules and preorderings}

For a finite $S\subset \smatpoly$ let
$$K_S=\left\{X\in\mat^g\colon s(X)\succeq0 \ \forall s\in S\right\}$$
be the {\bf semialgebraic set} described by $S$.
A set $\cQ\subseteq \STR$ is a {\bf cyclic quadratic module} if
$$1\in\cQ,\quad \cQ+\cQ\subseteq \cQ,
\quad h\cQ h^{\ti}\subseteq \cQ\ \ \forall h\in\TR,
\quad \tr(\cQ)\subset \cQ.$$
A cyclic quadratic module $\cT\subseteq \STR$ is a {\bf cyclic preordering} if $\cT\cap \cen$ is closed under multiplication. For $S\subset\STR$ let $\cQ^{\tr}_S$ and $\cT^{\tr}_S$ denote the cyclic quadratic module and preordering, respectively, generated by $S$. For example, $\cQ^{\tr}_\emptyset=\cT^{\tr}_\emptyset=\TP$.

\begin{lem}\label{l:gen}
Let $S\subseteq\STR$.
\begin{enumerate}[label={\rm(\arabic*)}]
	\item If $\cQ$ is a cyclic quadratic module, then $\tr(\cQ)=\cQ\cap \cen$.
	\item Elements of $\cQ^{\tr}_S$ are precisely sums of
	$$q_1,\qquad h_1s_1h_1^{\ti},\qquad \tr(h_2s_2h_2^{\ti})q_2$$
	for $q_i\in\TP$, $h_i\in\TR$ and $s_i\in S$.
	\item $\cT^{\tr}_S=\cQ^{\tr}_{S'}$, where
	$$S'=S\cup \left\{\prod_i\tr(h_is_ih_i^{\ti})\colon h_i\in\TR,s_i\in S\right\}.$$
\end{enumerate}
\end{lem}

\begin{proof}
Straightforward.
\end{proof}

Our main result of this subsection is a reduction to central generators for cyclic quadratic modules, see Corollary \ref{c:scal}. It will be used several times in the sequel. In its proof we need the following lemma.

\begin{lem}\label{l:dense}
Let $R$ be an ordered field, $\lambda_1,\dots,\lambda_n\in R$ and $p_i=\sum_{j=1}^n\lambda_j^i$ for $i\in\N$. If $\lambda_{j_0}<0$ for some $1\le j_0\le n$, then there exists $f\in \Q[p_1,\dots,p_n][\zeta]$ such that
\begin{equation}\label{e:neg}
\sum_{j=1}^n f(\lambda_j)^2\lambda_j<0.
\end{equation}
\end{lem}

\begin{proof}
Denote $E=\Q(p_1,\dots,p_n)$ and $F=\Q(\lambda_1,\dots,\lambda_n)$. For every $f=\sum_{i=0}^{n-1}\alpha_i\zeta^i\in F[\zeta]$ we have
\begin{equation}\label{e:expand}
\sum_{j=1}^n f(\lambda_j)^2\lambda_j
=\sum_j\sum_{i,i'} \alpha_i\alpha_{i'}\lambda_j^{i+i'+1}
=\sum_{i,i'}\left(\sum_j \lambda_j^{i+i'+1}\right)\alpha_i\alpha_{i'}
=\sum_{i,i'}p_{i+i'+1}\alpha_i\alpha_{i'}.
\end{equation}
Note that $p_i\in E$ for every $i\in\N$ and define $P\in \opm_n(E)$ by $P_{ij}=p_{i+j-1}$. If $\lambda_{j_0}<0$, then there clearly exists $f_0\in F[\zeta]$ of degree $n-1$ such that $f_0(\lambda_{j_0})\neq0$ and $f_0(\lambda_j)=0$ for $\lambda_j\neq\lambda_{j_0}$. Then $f_0$ satisfies \eqref{e:neg}, so $P$ is not positive semidefinite as a matrix over $F$ by \eqref{e:expand}. Since $P=QDQ^{\ti}$ for some $Q\in\GL_n(E)$ and diagonal $D\in\opm_n(E)$, we conclude that $P$ is not positive semidefinite as a matrix over $E$, so there exists $v=(\beta_0,\dots,\beta_{n-1})^{\ti}\in E^n$ such that $v^{\ti}Pv<0$. By \eqref{e:expand}, $f_1=\sum_{i=0}^{n-1}\beta_i\zeta^i\in E[\zeta]$ satisfies \eqref{e:neg}. After clearing the denominators of the coefficients of $f_1$ we obtain $f\in \Q[p_1,\dots,p_n][\zeta]$ satisfying \eqref{e:neg}.
\end{proof}

The proof of the next proposition requires some well-known notions and facts from real algebra that we recall now. Let $\Lambda$ be a commutative unital ring. Then $P\subset \Lambda$ is a {\bf ordering} if $P$ is closed under addition and multiplication, $P\cup -P=\Lambda$ and $P\cap -P$ is a prime ideal in $\Lambda$. Note that every ordering in $\Lambda$ gives rise to a ring homomorphism from $\Lambda$ into a real closed field and vice versa. The set of all orderings is the {\bf real spectrum} of $\Lambda$, denoted $\Sper\Lambda$. For $a\in\Lambda$ let $K(a)=\{P\in\Sper\Lambda\colon a\in P \}$. Then the sets $K(a)$ and $\Sper\Lambda\setminus K(a)$ for $a\in\Lambda$ form a subbasis of the {\bf constructible topology} \cite[Section 7.1]{BCR} (also called patch topology \cite[Section 2.4]{Mar}). By \cite[Proposition 1.1.12]{BCR} or \cite[Theorem 2.4.1]{Mar}, $\Sper\Lambda$ endowed with this topology is a compact Hausdorff space. In particular, since the sets $K(a)$ are closed in $\Sper\Lambda$, they are also compact.

\begin{prop}\label{p:scal}
For every $s\in\STR$ let $\cO\subset \TR$ be the ring of polynomials in $s$ and $\tr(s^i)$ for $i\in\N$ with rational coefficients, and set
$$S=\left\{ \tr(hsh)\colon h\in\cO\right\}\subset\tr\left(\cQ^{\tr}_{\{s\}}\right).$$
Then there exists a finite subset $S_0\subset S$ such that $K_{\{s\}}=K_{S_0}$.
\end{prop}

\begin{proof}
First we prove that for every real closed field $R$ we have
\begin{equation}\label{e:inf}
\left\{X\in\opm_n(R)^g\colon s(X)\succeq0 \right\} =
\bigcap_{c\in S} \left\{X\in\opm_n(R)^g\colon c(X)\ge0 \right\}.
\end{equation}
The inclusion $\subseteq$ is obvious. Let $X\in\opm_n(R)^g$ be such that $s(X)$ is not positive semidefinite. Since $R$ is real closed and pure trace polynomials are $\ort$-invariant, we can assume that $s(X)=\diag(\lambda_1,\dots,\lambda_n)$ is diagonal and $\lambda_j<0$ for some $j$. If $p_i=\tr(s(X)^i)$, then by Lemma \ref{l:dense} there exists a polynomial $f\in\Q[p_1,\dots,p_n][\zeta]$ such that
$$\sum_{i=1}^n f(\lambda_i)^2\lambda_i<0.$$
If $h\in\cO$ is such that $h(X)=f(s(X))$, then $\tr(h(X)s(X)h(X))<0$. Hence $\supseteq$ in \eqref{e:inf} holds.

Let $\sigma_j=\tr(\wedge^j s)\in\cen$ for $1\le j\le n$, where $\wedge^j s$ denotes the $j$th exterior power of $s$; hence $\sigma_j$ are signed coefficients of the characteristic polynomial for $s$ and
$$\left\{X\in\opm_n(R)^g\colon s(X)\succeq0 \right\}
=\left\{X\in\opm_n(R)^g\colon \sigma_1(X)\ge0,\dots,\sigma_n(X)\ge0 \right\}$$
for all real closed fields $R$. In terms of $\Sper\R[\ulxi]$ and the notation introduced before the proposition, \eqref{e:inf} can be stated as
\begin{equation}\label{e:inf2}
\bigcap_{j=1}^n K(\sigma_j) =\bigcap_{c\in S} K(c)
\end{equation}
by the correspondence between homomorphisms from $\R[\ulxi]$ to real closed fields and orderings in $\R[\ulxi]$. Since the 
complement of the left-hand side of \eqref{e:inf2} is compact in the constructible topology, there exists a finite subset $S_0\subset S$ such that
$$\bigcap_{j=1}^n K(\sigma_j) =\bigcap_{c\in S_0} K(c)$$
and consequently
\[K_{\{s\}}=K_{\{\sigma_1,\dots,\sigma_n\}}=K_{S_0}.\qedhere\]
\end{proof}

\begin{cor}\label{c:scal}
For every finite set $S\subset\STR$ there exists a finite set $S'\subset \tr(\cQ^{\tr}_S)$ such that $K_S=K_{S'}$.
\end{cor}

\begin{proof}
Let $S=\{s_1,\dots,s_\ell\}$. By Proposition \ref{p:scal} there exist finite sets $S_i\subset \cQ^{\tr}_{\{s_i\}}\cap \cen$ with $K_{S_i}=K_{\{s_i\}}$. If $S'=S_1\cup\cdots\cup S_\ell$, then
$$S'\subset \bigcup_i \cQ^{\tr}_{\{s_i\}}\cap \cen\subset \cQ^{\tr}_S\cap \cen$$
and
\[K_{S'}=\bigcap_i K_{S_i}=\bigcap_i K_{\{s_i\}}=K_S. \qedhere\]
\end{proof}

\begin{cor}\label{c:scalcqm}
For every cyclic quadratic module $\cQ\subseteq\STR$ we have $K_{\cQ}=K_{\tr(\cQ)}$.
\end{cor}

\begin{proof}
Direct consequence of Corollary \ref{c:scal}.
\end{proof}

\subsection{An extension theorem}

The main result in this subsection, Theorem \ref{t:hom}, characterizes homomorphisms from pure trace polynomials $\cen$ to a real closed field $R$ which arise via point evaluations $\xi_{j\imath\jmath}\mapsto\alpha_{j\imath\jmath}\in R$.

We start with some additional terminology. Let $F$ be a field and let $\cA$ be a finite-dimensional simple $F$-algebra with center $C$. If $\tr_{\cA}$ is the reduced trace of $\cA$ as a central simple algebra and $\tr_{C/F}$ is the trace of the field extension $C/F$, then
$$\tr_{\cA}^F=\tr_{C/F}\circ\tr_{\cA}:\cA\to F$$
is called the {\bf reduced $F$-trace} of $\cA$ \cite[Section 4]{DPRR}.

\begin{prop}\label{p:posinv}
Let $R\supseteq\R$ be a real closed field and $\cA$ a finite-dimensional semisimple $R$-algebra with an $R$-trace $\chi$ and a split involution, which is positive on every simple factor. Assume there exists a trace preserving $*$-homomorphism $\Phi:\TR\to\cA$ such that $\Phi(\cen)\subseteq R$ and $\cA$ is generated by $\Phi(\TR)$ over $R$. Then there exists a trace preserving $*$-embedding of $\cA$ into $(\opm_n(R),\ti,\tr)$.
\end{prop}

\begin{proof}
Let $\cA\cong\cA_1\times\cdots\times\cA_\ell$ be the decomposition in simple factors and let $n_k$ be the degree of $\cA_k$ for $1\le k\le \ell$. Moreover, let $C=R[i]$ be the algebraic closure of $R$ and $H=(\frac{-1,-1}{R})$ the division quaternion algebra over $R$. By \cite[Theorem 1.2]{PS}, each of $\cA_k$ is $*$-isomorphic to one of the following:
\begin{enumerate}[label={\rm(\Roman*)}]
	\item $\opm_{n_k}(R)$ with the transpose involution;
	\item $\opm_{n_k}(C)$ with the conjugate-transpose involution;
	\item $\opm_{n_k/2}(H)$ with the symplectic involution.
\end{enumerate}
Without loss of generality assume that there are $1\le\ell_1\le\ell_2\le\ell$ such that $\cA_{n_k}$ is of type (I) for $k\le\ell_1$, of type (II) for $\ell_1<k\le \ell_2$, and of type (III) for $\ell_2<k$. By \cite[Theorem 4.2]{DPRR} there exist $d_1,\dots,d_\ell\in\N$ such that
\begin{equation}\label{e:trdecomp0}
\chi\left(\sum_k a_k\right)=\sum_k d_k\tr_{\cA_k}^R(a_k)
\end{equation}
for $a_k\in\cA_k$. We claim that $d_k\in 2\N$ for every $k>\ell_2$. Let $n'=\lceil\frac{n}{2}\rceil$ and fix $k>\ell_2$. By Lemma \ref{l:skew0}, $f=f_{n'}(x_1-x_1^*,x_2-x_2^*)$ is a $*$-trace identity for $\TR$. Therefore $f$ is also a $*$-trace identity for $\cA$ by the assumptions on $\Phi$. Hence $f$ is a $*$-trace identity for $(\cA_k,\tau_k,d_k\cdot\tr_{\cA_k}^R)$, where $\tau_k$ is the restriction of the involution on $\cA$. Since $*$-trace identities are preserved by scalar extensions and
$$C\otimes_R\left(\cA_k,\tau_k,d_k\cdot\tr_{\cA_k}^R\right)
\cong C\otimes_R\left(\opm_{n_k/2}(H),\si,d_k\cdot(\tr_H\circ\tr)\right)
\cong \left(\opm_{n_k}(C),\si,d_k\cdot\tr\right),$$
$f$ is a $*$-trace identity for $(\opm_{n_k}(C),\si,d_k\cdot\tr)$. Now Lemma $\ref{l:skew}$ implies $d_k\in2\N$.

Thus we have
\begin{equation}\label{e:rel}
n=\sum_{k\le \ell_1}d_kn_k+\sum_{\ell_1<k\le \ell_2}2d_kn_k+\sum_{\ell_2<k}4\frac{d_k}{2}n_k
\end{equation}
by \eqref{e:trdecomp0} and the definition of the reduced $R$-trace. The standard embeddings
\begin{alignat*}{4}
\psi_1:C & \hookrightarrow \opm_2(R), \qquad &\alpha+\beta i & \mapsto &&
\begin{pmatrix}\alpha & -\beta \\ \beta & \alpha \end{pmatrix} \\
\psi_2:H & \hookrightarrow\opm_4(R),\qquad &\alpha+\beta i+\gamma j+\delta k & \mapsto &&
\begin{pmatrix}
\alpha & -\beta & -\gamma & -\delta \\
\beta & \alpha & -\delta & \gamma \\
\gamma & \delta & \alpha & -\beta \\
\delta & -\gamma & \beta & \alpha
 \end{pmatrix}
\end{alignat*}
transform conjugate-transpose involution and symplectic involution into transpose involution; moreover, $\psi_1$ preserves the reduced $R$-trace, while $\psi_2$ doubles it. Therefore we have trace preserving $*$-embeddings
$$\left(\opm_{n_k}(R),\ti,d_k\cdot\tr\right)\hookrightarrow \left(\opm_{d_kn_k}(R),\ti,\tr\right),
\qquad X\mapsto X^{\oplus d_k}$$
for $k\le \ell_1$,
$$\left(\opm_{n_k}(C),*,d_k\cdot(\tr_{C/R}\circ\tr)\right)\hookrightarrow  \left(\opm_{2d_kn_k}(R),\ti,\tr\right),
\qquad X\mapsto \psi_1(X)^{\oplus d_k}$$
for $\ell_1<k\le \ell_2$, and
$$\left(\opm_{n_k/2}(H),\si,d_k\cdot(\tr_H\circ\tr)\right)\hookrightarrow
\left(\opm_{2d_kn_k}(R),\ti,\tr\right),
\qquad X\mapsto \psi_2(X)^{\oplus d_k/2}$$
for $\ell_2<k$, where $\psi_1$ and $\psi_2$ are applied entry-wise. By \eqref{e:rel} we can combine these embeddings to obtain a trace preserving $*$-embedding
\begin{align*}
\cA \xrightarrow{\cong} &
\cA_1\times\cdots\times \cA_\ell \\
\xrightarrow{\cong} & \prod_{k\le \ell_1} \opm_{n_k}(R)
\times \prod_{\ell_1< k\le \ell_2} \opm_{n_k}(C)
\times \prod_{\ell_2< k} \opm_{n_k/2}(H) \\
\hookrightarrow & \prod_{k\le \ell_1} \opm_{d_kn_k}(R)
\times \prod_{\ell_1< k\le \ell_2} \opm_{2d_kn_k}(R)
\times \prod_{\ell_2< k} \opm_{2d_kn_k}(R) \\
\hookrightarrow & \left(\opm_n(R),\ti,\tr\right).\qedhere
\end{align*}
\end{proof}

\begin{lem}\label{l:tech}
Let $\Lambda$ be a Noetherian domain with $\kar\Lambda\neq2$, $M$ a finitely generated $\Lambda$-module, $K$ a field, $\phi:\Lambda\to K$ a ring homomorphism, and $b:M\times M\to\Lambda$ a symmetric $\Lambda$-bilinear form. Let $\pi:M\to K\otimes_\phi M$ be the natural homomorphism. Then there exist $u_1,\dots,u_\ell\in M$ such that
$\{\pi(u_1),\dots,\pi(u_\ell)\}$ is a $K$-basis of $K\otimes_\phi M$ and $\phi(b(u_i,u_{i'}))=0$ for $i\neq i'$.
\end{lem}

\begin{proof}
Let $\ell=\dim_K(K\otimes_\phi M)$. We prove the statement by induction on $\ell$. The case $\ell=1$ is trivial. Now assume that statement holds for $\ell-1$ and suppose $K\otimes_\phi M$ is of dimension $\ell$. 

If $\phi\circ b=0$, we are done. Otherwise there exists $u_1\in M\setminus\ker\pi$ with $b(u_1,u_1)\notin\ker\phi$. Indeed, if $\phi(b(u,u))=0$ for all $u\in M\setminus\ker\pi$, then $\phi(b(u,u))=0$ for all $u\in M$, so by
$$2b(u,v)=b(u+v,u+v)-b(u,u)-b(v,v)$$
it follows that $\phi(b(u,v))=0$ for every $u,v\in M$. Clearly there exist $v_2,\dots,v_\ell\in M$ such that $\{\pi(u_1),\pi(v_2)\dots,\pi(v_\ell)\}$ is a $K$-basis of $K\otimes_\phi M$. For $2\le i\le \ell$ let
$$v_i'=b(u_1,u_1)v_i-b(u_1,v_i)u_1$$
and let $M'$ be the $\Lambda$-module generated by $v_i'$. Note that $b(u_1,v)=0$ for all $v\in M'$ and $\dim_K(K\otimes_\phi M')=\ell-1$ since $\phi(b(u_1,u_1))$ is invertible in $K$. Hence we can apply the induction hypothesis to obtain $u_2,\dots,u_\ell\in M'$ such that $\{\pi(u_1),\dots,\pi(u_\ell)\}$ is a $K$-basis of $K\otimes_\phi M$ and $\phi(b(u_i,u_{i'}))=0$ for all $i\neq i'$.
\end{proof}

\begin{thm}\label{t:hom}
Let $R\supseteq\R$ be a real closed field. Then an $\R$-algebra homomorphism $\phi:\cen\to R$ extends to an $\R$-algebra homomorphism $\varphi:\R[\ulxi]\to R$ if and only if $\phi(\TPc)\subseteq R_{\ge0}$.
\end{thm}

\begin{proof}
The implication $(\Rightarrow)$ is obvious, so we prove $(\Leftarrow)$. In the terminology of \cite[Subsection 2.3]{DPRR}, $\TR$ is an $n$--Cayley-Hamilton algebra. Since $\TR$ is finitely spanned over $\cen$, $\cA'=R\otimes_\phi\TR$ is a finite-dimensional $R$-algebra which inherits an involution $\tau'$ and an $R$-trace $\chi':\cA'\to R$ from $\TR$. By \cite[Subsection 2.3]{DPRR} $\cA'$ is again an $n$--Cayley-Hamilton algebra. Let $\cJ$ be the Jacobson radical of $\cA'$. Since $\cA'$ is finite-dimensional, elements of $\cJ$ are characterized as generators of nilpotent ideals. Hence clearly $\cJ^{\tau'}\subseteq\cJ$. Moreover, if $f\in\cJ$, then $\chi'(f)=0$ by applying \cite[Proposition 3.2]{DPRR} to the scalar extension of $\cA'$ by the algebraic closure of $R$ and \cite[Theorem 5.17]{Lam}.

Therefore $\cA=\cA'/\cJ$ is a finite-dimensional semisimple $R$-algebra with involution $\tau$ and an $R$-trace $\chi:\cA\to R$. If $\Phi:\TR\to\cA$ is the canonical $*$-homomorphism, then
\begin{equation}\label{e:tr}
\chi\circ \Phi=\Phi\circ\tr.
\end{equation}
We claim that $\tau(aa^\tau)\ge0$ for every $a\in\cA$. Indeed, if $\pi:\TR\to R\otimes_\phi \TR$ is the canonical $*$-homomorphism, then by Lemma \ref{l:tech} there exist a finite set $\{u_i\}_i$ of symmetric elements in $\TR$ and a finite set $\{v_j\}_j$ of antisymmetric elements in $\TR$ such that $\{\pi(u_i)\}_i\cup\{\pi(v_j)\}_j$ form an $R$-basis of $R\otimes_\phi\TR$ and
$$\phi(\tr(u_iu_{i'}))=\phi(\tr(v_jv_{j'}))=\phi(\tr(u_iv_j))=0$$
for all $i\neq i'$ and $j\neq j'$. If
$$a=\sum_i\alpha_i\Phi(u_i)+\sum_j\beta_j\Phi(v_j),\qquad \alpha_i,\beta_j\in R,$$
then
$$\chi(aa^\tau)=\sum_i\alpha_i^2\phi(\tr(u_iu_i^{\ti}))+\sum_j\beta_j^2\phi(\tr(v_jv_j^{\ti}))\ge0$$
by \eqref{e:tr}.

By Wedderburn's structure theorem we have
$$\cA=\cA_1\times\cdots \times\cA_\ell$$
for some finite-dimensional simple $R$-algebras $\cA_k$. Moreover, by \cite[Theorem 4.2]{DPRR} there exist $d_1,\dots,d_\ell\in\N$ such that
\begin{equation}\label{e:trdecomp}
\chi\left(\sum_k a_k\right)=\sum_k d_k\tr_{\cA_k}^R(a_k)
\end{equation}
for $a_k\in\cA_k$.

Next we show that $\tau$ is split, i.e., $(\cA_k)^\tau\subseteq\cA_k$ for $1\le k\le \ell$. Since every involution preserves centrally primitive idempotents \cite[Section 22]{Lam}, for every $k$ there exists $k'$ such that $(\cA_k)^\tau\subseteq\cA_{k'}$. Suppose that $\tau$ is not split and without loss of generality assume $(\cA_1)^\tau\subseteq\cA_2$. Let $e_1\in\cA_1$ and $e_2\in\cA_2$ be the identity elements, respectively. Then
\begin{align*}
\chi\big((e_1-e_2)(e_1-e_2)^\tau\big)
&=\chi\big((e_1-e_2)(e_2-e_1)\big) \\
&=\chi(-e_1-e_2) \\
&=-d_1\tr_{\cA_1}^R(e_1)-d_2\tr_{\cA_2}^R(e_2)<0,
\end{align*}
a contradiction.

Let $\tau_k$ be the restriction of $\tau$ on $\cA_k$. By \eqref{e:trdecomp} and the previous paragraph it follows that $\tr_{\cA_k}^R(aa^{\tau_k})\ge0$ and hence $\tr_{\cA_k}(aa^{\tau_k})\ge0$ for every $a\in \cA_k$, so $\tau_k$ is a positive involution.

Therefore the assumptions of Proposition \ref{p:posinv} are met and we obtain a trace preserving $*$-homomorphism $\Psi:\TR\to\opm_n(R)$ extending $\phi$. Now we define $\varphi:\R[\ulxi]\to R$ by
\[\varphi(\xi_{j\imath\jmath})=\Psi(\Xi_j)_{\imath\jmath}. \qedhere\]
\end{proof}

\begin{rem}
The condition $\phi(\TPc)\subseteq R_{\ge0}$ in Theorem \ref{t:hom} is clearly necessary since $\tr(hh^*)$ is a nonzero sum of squares in $\R[\ulxi]$ for every nonzero $h\in\TR$. Moreover, it is not vacuous. For example, let $\tau$ be the involution on $\HH$ defined by $u^\tau=iu^{\si}i^{-1}$ for $u\in\HH$, where $\si$ is the standard symplectic involution on $\HH$. Then $\tau$ is of orthogonal type and we have a trace preserving $*$-epimorphism $\Phi:\TT_2\to\HH$ defined by $\Phi(\Xi_1)=i$ and $\Phi(\Xi_2)=j$. Since $\Phi(T_2)=\R$, the restriction yields a homomorphism $\phi:T_2\to\R$ and $\phi(\tr(\Xi_2\Xi_2^{\ti}))=-1$.
\end{rem}

\begin{cor}\label{c:pslike}
Let $\phi:\cen=\R[\ulxi]^{\ort}\to R$ be an $\R$-algebra homomorphism into a real closed field $R\supseteq\R$. Then the following are equivalent:
\begin{enumerate}[label={\rm(\roman*)}]
	\item $\phi$ extends to an $\R$-algebra homomorphism $\varphi:\R[\ulxi]\to R$;
	\item $\phi(\R[\ulxi]^{\ort}\cap \sum \R[\ulxi]^2)\subseteq R_{\ge0}$;
	\item $\phi(\tr(hh^\ti))\in R_{\ge0}$ for all $h\in\TR$.
\end{enumerate}
\end{cor}

\begin{rem}\label{r:noPS}
At first glance one might ponder whether Theorem \ref{t:hom} could be derived from the Procesi-Schwarz theorem \cite{PS1}. In Appendix \ref{a:ps} we explain why this does not seem to be the case.
\end{rem}

\subsection{Stellens\"atze}

We are now ready to give the main result of this section, the Krivine-Stengle Positivstellensatz for trace polynomials $a$ that are positive (semidefinite) on $K_S$, see Theorem \ref{t:posss}. In the proof we use Corollary \ref{c:scal} to reduce the problem to the commutative ring $\cen[a]$. Before applying the abstract Positivstellensatz for commutative rings, we need the relation between orderings and matrix evaluations of trace polynomials that is given in Proposition \ref{p:ord} below, which is a crucial consequence of the extension Theorem \ref{t:hom} and Tarski's transfer principle.

\begin{prop}\label{p:ord}
Let $S\subset \cen$ be finite, $a\in\STR$ and $P$ an ordering in $\cen[a]$ containing $S\cup\TPc$.
\begin{enumerate}[label={\rm(\arabic*)}]
	\item $a|_{K_S}\succeq 0$ implies $a\in P$.
	\item $a|_{K_S}\succ 0$ implies $a\in P\setminus-P$.
	\item $a|_{K_S}= 0$ implies $a\in P\cap\ -P$.
\end{enumerate}
\end{prop}

\begin{proof}
Let $P$ be an ordering in $\cen[a]$ containing $S$ and let $\sigma_j=\tr(\wedge^j a)\in\cen$ for $1\le j\le n$.

(1) The restriction of $P$ to $\cen$ gives rise to a real closed field $R$ and a homomorphism $\phi:\cen\to R$ satisfying $\phi(S\cup\TPc)\subseteq R_{\ge0}$. By Theorem \ref{t:hom} we extend it to a homomorphism $\varphi:\R[\ulxi]\to R$. Suppose that $\varphi(\sigma_j)<0$ for some $j$; in other words, there exist $\alpha\in R^{gn^2}$ such that $\sigma_j(\alpha)<0$ and $s'(\alpha)\ge0$ for all $s'\in S$. By Tarski's transfer principle \cite[Theorem 1.4.2]{Mar} there exists $\alpha'\in\R^{gn^2}$ such that $\sigma_j(\alpha')<0$ and $s'(\alpha')\ge0$ for all $s'\in S$. But this contradicts $\sigma_j|_{K_S}\ge 0$, which is a consequence of $s|_{K_S}\succeq 0$. Hence $\phi(\sigma_j)=\varphi(\sigma_j)\ge0$ for all $j$, so $\sigma_j\in P$ for all $j$. By the Cayley-Hamilton theorem we have
\begin{equation}\label{e:ch}
(-a)^n+\sum_{j=1}^n\sigma_j(-a)^{n-j}=0.
\end{equation}
Suppose $a\notin P$. Then $-a\in P$, so \eqref{e:ch} implies $(-a)^n\in P\cap- P$. Therefore $a\in P\cap -P$, a contradiction.

(2) Because $a|_{K_S}\succ 0$ implies $\sigma_j|_{K_S}> 0$, we obtain $\sigma_j\in P\setminus -P$ for all $j$ by the same reasoning as in (1). If $a\notin P\setminus -P$, then $-a\in P$, so \eqref{e:ch} implies $\sigma_n\in P\cap- P$, a contradiction.

(3) If $a|_{K_S}= 0$, then $a|_{K_S}\succeq0$ and $-a|_{K_S}\succeq0$, so $a\in P\cap -P$ by (1).
\end{proof}

\begin{thm}[Krivine-Stengle Positivstellensatz for trace polynomials]\label{t:posss}
Let $S\cup\{a\}\subset \STR$ be finite.
\begin{enumerate}[label={\rm(\arabic*)}]
	\item $a|_{K_S}\succeq 0$ if and only if $at_1=t_1a=a^{2k}+t_2$ for some $t_1,t_2\in\cT^{\tr}_S$ and $k\in\N$.
	\item $a|_{K_S}\succ 0$ if and only if $at_1=t_1a=1+t_2$ for some $t_1,t_2\in\cT^{\tr}_S$.
	\item $a|_{K_S}= 0$ if and only if $-a^{2k}\in\cT^{\tr}_S$ for some $k\in\N$.
\end{enumerate}
\end{thm}

\begin{proof}
The directions $(\Leftarrow)$ are straightforward. For the implications $(\Rightarrow)$, by Corollary \ref{c:scal} we can assume that $S\subset\cen$. Let $T$ be the preordering in $\cen[a]$ generated by $S\cup\TPc$. Note that $T\subset \cT^{\tr}_S$ since $S\subset\cen$. If $a|_{K_S}\succeq 0$, then $a\in P$ for every ordering $P$ of $\cen[a]$ containing $T$ by Proposition \ref{p:ord}. Therefore $t_1a=a^{2k}+t_2$ for some $t_1,t_2\in T$ and $k\in\N$ by the abstract Positivstellensatz \cite[Theorem 2.5.2]{Mar}, so (1) is holds. (2) and (3) are proved analogously.
\end{proof}

\begin{rem}
In general we cannot choose a central $t_1$ in Theorem \ref{t:posss}; see Example \ref{exa:cent}.
\end{rem}

\begin{rem}
A clean Krivine-Stengle Positivstellensatz for generic matrices clearly does not exist for $n=3$ due to Theorem \ref{t:ps3}. Moreover, in Example \ref{exa:gm2} we show that even for $n=2$, where Conjecture \ref{co:ps} holds, the traceless equivalent of Theorem \ref{t:posss} for $\GM$ fails.
\end{rem}

\begin{cor}\label{c:empty}
If $S\subset\STR$ is finite, then $K_S=\emptyset$ if and only if $-1\in\cT^{\tr}_S$.
\end{cor}

\begin{proof}
If $K_S=\emptyset$, then $-1|_{K_S}\succ0$, so by Theorem \ref{t:posss} there exist $t_1,t_2\in \cT^{\tr}_S$ such that $(-1)t_1=1+t_2$, so $-1=t_1+t_2\in\cT^{\tr}_S$. The converse is trivial.
\end{proof}

\begin{cor}\label{c:empty1}
Let $s\in\STR$. Then $s(A)\not\succeq0$ for all $A\in\mat^g$ if and only if
$$-1=\sum_i\omega_i\prod_j\tr(h_{ij}sh_{ij}^{\ti})$$
for some $\omega_i\in\TPc$ and $h_{ij}\in\TR$.
\end{cor}

\begin{proof}
Follows by Corollary \ref{c:empty} and Lemma \ref{l:gen}.
\end{proof}

\begin{cor}[Real Nullstellensatz for trace polynomials]\label{c:renull}
Let $\cJ\subset\TR$ be an ideal and assume $\tr(\cJ)\subseteq\cJ$. For $h\in\TR$ the following are equivalent:
\begin{enumerate}[label={\rm(\roman*)}]
	\item for every $X\in\mat^g$, $u(X)=0$ for every $u\in\cJ$ implies $h(X)=0$;
	\item there exists $k\in\N$ such that $-(h^{\ti}h)^k\in \TP+\cJ$.
\end{enumerate}
\end{cor}

\begin{proof}
The implication $(2)\Rightarrow(1)$ is clear. Conversely, $\TR$ is Noetherian, so $\cJ$ is (as a left ideal) generated by some $u_1,\dots,u_\ell\in \TR$. Let $S=\{-u_1^{\ti}u_1,\dots,-u_\ell^{\ti}u_\ell \}$; then (1) is equivalent to $h^{\ti}h|_{K_S}=0$. Hence $(1)\Rightarrow(2)$ follows by Theorem \ref{t:posss}(3) and $\cT^{\tr}_S\subseteq \TP+\cJ$.
\end{proof}

\begin{cor}\label{c:renull1}
Let $S\subset\TR$. Then
$$\{A\in\mat^g\colon s(A)=0\ \forall s\in S\}=\emptyset$$
if and only if
$$-1=\omega+\sum_i\tr(h_is_i)$$
for some $\omega\in\TPc$, $h_i\in\TR$ and $s_i\in S$.
\end{cor}

We mention that Hilbert's Nullstellensatz for $n\times n$ generic matrices over an algebraically closed field is given by Amitsur in \cite[Theorem 1]{Ami}.

%%%%%%%%%%%%%%%%%%%%%%%%%%%%%%%%%
%%%%%%%%%%%%%%%%%%%%%%%%%%%%%%%%%

\section{Positivstellens\"atze for compact semialgebraic sets}\label{s:comp}

In this section we give certificates for positivity on compact semialgebraic sets. We prove a version of Schm\"udgen's theorem \cite{Schm} for trace polynomials (Theorem \ref{t:sch}). We also present a version of Putinar's theorem \cite{Put} for trace polynomials (Theorem \ref{t:arch}) and for generic matrices (Theorem \ref{t:arch1}).

\subsection{Archimedean (cyclic) quadratic modules}

A cyclic quadratic module $\cQ\subset\TR$ is {\bf archimedean} if for every $h\in\TR$ there exists $\rho\in\Q_{>0}$ such that $\rho-hh^{\ti}\in\cQ$. Equivalently, for every $s\in \STR$ there exists $\varepsilon\in\Q_{>0}$ such that $1\pm\varepsilon s\in \cQ$.

For a cyclic quadratic module $\cQ$ let $H_{\cQ}$ be the set of elements $h\in\TR$ such that $\rho-hh^{\ti}\in\cQ$ for some $\rho\in\Q_{>0}$. It is clear that $\cQ$ is archimedean if and only if $H_{\cQ}=\TR$.

\begin{prop}\label{p:hq}
$H_{\cQ}$ is a trace $*$-subalgebra over $\R$ in $\TR$.
\end{prop}

\begin{proof}
$H_{\cQ}$ is a $*$-subalgebra over $\R$ by \cite{Vid}. Let $h\in H_{\cQ}$. Then $s=h+h^{\ti}\in H_{\cQ}$ and let $\rho\in\Q_{>0}$ be such that $\rho^2-s^2\in\cQ$. Then
$$\rho\pm s=\frac{1}{2\rho}\left((\rho\pm s)^2+(\rho^2-s^2)\right)\in\cQ$$
and consequently $n\rho\pm\tr(s)\in\cQ$. Therefore
$$(n\rho)^2-\tr(s)^2=\frac{1}{2n\rho}\big((\rho-s)(\rho+s)(\rho-s)+(\rho+s)(\rho-s)(\rho+s)\big)\in\cQ,$$
so $\tr(h)=\frac12 \tr(s)\in H_{\cQ}$.
\end{proof}

\begin{cor}\label{c:arch1}
A cyclic quadratic module $\cQ$ is archimedean if and only if there exists $\rho\in\Q_{>0}$ such that $\rho-\sum_j\Xi_j\Xi_j^{\ti}\in\cQ$.
\end{cor}

\begin{proof}
$(\Rightarrow)$ is trivial. Conversely, $\rho-\sum_j\Xi_j\Xi_j^{\ti}\in\cQ$ implies $\Xi_j\in H_{\cQ}$ for $1\le j\le n$, so $H_{\cQ}=\TR$ since $\TR$ is generated by $\Xi_j$ as a trace $*$-subalgebra over $\R$ by Proposition \ref{p:hq}.
\end{proof}

It is easy to see that $K_{\cQ}$ is compact if $\cQ$ is archimedean. The converse fails already with $n=1$ (\cite[Section 7.3]{Mar} or \cite[Example 6.3.1]{PD}). If $K_{\cQ}$ is compact, say $\|X\|\le N$ for all $X\in K_{\cQ}$, then we can add $N^2-\sum_j\Xi_j\Xi_j^{\ti}$ to $\cQ$ to make it archimedean without changing $K_{\cQ}$.

\subsection{Schm\"udgen's Positivstellensatz for trace polynomials}

In this subsection we prove a version of Schm\"udgen's Positivstellensatz for trace polynomials $a$ that are positive on a compact semialgebraic set $K_S$. The proof is a two-step commutative reduction. Firstly, the constraints $S$ are replaced with central ones by Corollary \ref{c:scal}. Then the abstract version of Schm\"udgen's Positivstellensatz is used in the commutative ring $\cen[a]$.

\begin{thm}[Schm\"udgen's Positivstellensatz for trace polynomials]\label{t:sch}
Let $S\cup\{a\}\subset\STR$ be finite. If $K_S$ is compact and $a|_{K_S}\succ0$, then $a\in\cT_S^{\tr}$.
\end{thm}

\begin{proof}
First we apply Corollary \ref{c:scal} to reduce to the case $S\subset \cen$. Let $T$ be the preordering in $\cen[a]$ generated by $S\cup\TPc$. Note that $K_S=K_T$ and $T\subset \cT^{\tr}_S$.

Let $b\in \cen[a]$ be arbitrary. Since $K_S$ is compact, there exists $\beta\in\R_{\ge0}$ such that $\beta\pm b |_{K_S}\succeq0$. Then $\beta\pm b\in P$ for every ordering $P$ in $\cen[a]$ containing $T\supset S\cup\TPc$ by Proposition \ref{p:ord}. In the terminology of \cite{Sch}, $T$ is weakly archimedean. Since $\cen[a]$ is a finitely generated $\R$-algebra, $T$ is an archimedean preordering in $\cen[a]$ by the abstract version of Schm\"udgen's Positivstellensatz \cite[Theorem 3.6]{Sch}. Similarly, Proposition \ref{p:ord} implies $a\in P\setminus -P$ for every ordering $P$ in $\cen[a]$ containing $T\supset S\cup\TPc$, so $a\in T$ by \cite[Proposition 3.3]{Sch} or \cite[Theorem 4.3]{Mon}.
\end{proof}

\begin{cor}\label{c:sch}
Let $S\subset\STR$ be finite. Then $\cT_S^{\tr}$ is archimedean if and only if $K_S$ is compact.
\end{cor}

\subsection{Putinar's Positivstellensatz for generic matrices}

Our next theorem is a Putinar-type Positivstellensatz for generic matrices on compact semialgebraic sets, which requires a functional analytic proof. While the proof generally follows a standard outline (using a separation argument followed by a Gelfand-Naimark-Segal construction), several modifications are needed. For instance, the separation is taken to be extreme in a convex sense, and polynomial identities techniques are applied to produce $n\times n$ matrices.

A set $\cQ\subseteq \SGM$ is a {\bf quadratic module} if
$$1\in\cQ, \quad \cQ+\cQ\subseteq \cQ, \quad \quad h\cQ h^{\ti}\subseteq \cQ\ \ \forall h\in\GM.$$
We say that $\cQ$ is {\bf archimedean} if for every $h\in\GM$ there exists $\rho\in\Q_{>0}$ such that $\rho-aa^{\ti}\in\cQ$. As in Corollary \ref{c:arch1} we see that a quadratic module is archimedean if and only if it contains $\rho-\sum_j\Xi_j\Xi_j^{\ti}$ for some $\rho\in\Q_{>0}$.

\begin{thm}[Putinar's Positivstellensatz for generic matrices]\label{t:arch1}
Let $\cQ\subset\SGM$ be an ar\-chi\-medean quadratic module and $a\in\SGM$. If $a|_{K_{\cQ}}\succ0$, then $a\in \cQ$.
\end{thm}

\begin{proof}
Assume $a\in \SGM\setminus\cQ$. We proceed in several steps.

\smallskip {\sc Step 1:} Separation.\\
Consider $\cQ$ as a convex cone in the vector space $\SGM$ over $\R$. Since $\cQ$ is archimedean, for every $s\in\STR$ there exists $\varepsilon\in\Q_{>0}$ such that $1\pm\varepsilon s\in\cQ$, which in terms of \cite[Definition III.1.6]{Bar} means that $1$ is an algebraic interior point of the cone $\cQ$ in $\STR$. By the Eidelheit-Kakutani separation theorem \cite[Corollary III.1.7]{Bar} there exists a nonzero $\R$-linear functional $L_0:\SGM\to\R$ satisfying $L_0(\cQ)\subseteq\R_{\ge0}$ and $L_0(a)\le 0$. Moreover, $L_0(1)>0$ because $\cQ$ is archimedean, so after rescaling we can assume $L_0(1)=1$. Let $L:\GM\to\R$ be the symmetric extension of $L_0$, i.e., $L(f)=\frac12 L_0(f+f^{\ti})$ for $f\in\GM$.

\smallskip {\sc Step 2:} Extreme separation.\\
Now consider the set $\cC$ of all linear functionals $L':\GM\to \R$ satisfying $L'(\cQ)\subseteq\R_{\ge0}$ and $L'(1)=1$. This set is nonempty because $L\in\cC$. Endow $\GM$ with the norm
$$\|p\|=\max\left\{\|p(X)\|_2\colon X\in \mat^g, \|X\|_2\le1\right\}.$$
By the Banach-Alaoglu theorem \cite[Theorem III.2.9]{Bar}, the convex set $\cC$ is weak*-compact. Thus by the Krein-Milman theorem \cite[Theorem III.4.1]{Bar} we may assume that our separating functional $L$ is an extreme point of $\cC$.

\smallskip {\sc Step 3:} GNS construction.\\
On $\GM$ we define a semi-scalar product $\langle p, q\rangle=L(q^{\ti} p)$. By the Cauchy-Schwarz inequality for semi-scalar products,
$$\cN=\left\{q \in \GM \mid L(q^{\ti}q)=0\right\}$$
is a linear subspace of $\GM$. Hence
\begin{equation}\label{e:gns}
\langle\overline p,\overline q\rangle=L(q^{\ti} p)
\end{equation}
is a scalar product on $\GM/\cN$, where $\overline p=p+\cN$ denotes the residue class of $p\in\GM$ in $\GM/\cN$. Let $H$ denote the completion of $\GM/\cN$ with respect to this scalar product. Since $1\not\in \cN$, $H$ is non-trivial.

Next we show that $\cN$ is a left ideal of $\GM$. Let $p,q\in\GM$. Since $\cQ$ is archimedean, there exists $\varepsilon>0$ such that $1-\varepsilon p^{\ti}p\in \cQ$ and therefore
\begin{equation}\label{e:ineq}
0\le L(q^{\ti}(1-\varepsilon p^{\ti}p)q)\le L(q^{\ti}q).
\end{equation}
Hence $q\in\cN$ implies $pq\in\cN$.

Because $\cN$ is a left ideal, we can define linear maps
$$M_p :\GM/\cN\to\GM/\cN,\qquad\overline q\mapsto \overline{pq}$$
for $p\in\GM$. By \eqref{e:ineq}, $M_p$ is bounded and thus extends to a bounded operator $\hat M_p$ on $H$.

\smallskip {\sc Step 4:} Irreducible representation of $\GM$.\\
The map
$$\pi: \GM \to \cB(H), \qquad p\mapsto \hat M_p$$
is clearly a $*$-representation, where $\cB(H)$ is endowed with the adjoint involution $*$. Observe that $\eta=\overline 1\in H$ is a cyclic vector for $\pi$ by construction and
\begin{equation}\label{e:pil}
L(p)=\langle \pi(p)\eta,\eta\rangle.
\end{equation}
Write $\cA=\pi(\GM)$. We claim that the self-adjoint elements in the commutant $\cA'$ of $\cA$ in $\cB(H)$ are precisely real scalar operators. Let $P\in \cA'$ be self-adjoint. By the spectral theorem, $P$ decomposes into real scalar multiples of projections belonging to $\{P\}''\subseteq \cA'$. So it suffices to assume that $P$ is a projection. By way of contradiction suppose that $P\notin \{0,1\}$; since $\eta$ is cyclic for $\pi$, we have $P\eta\neq0$ and $(1-P)\eta\neq0$. Hence we can define linear functionals $L_j$ on $\GM$ by
$$L_1(p)=\frac{\langle\pi(p)P\eta,P\eta\rangle}{\|P\eta\|^2}
\qquad\text{and}\qquad
L_2(p)=\frac{\langle\pi(p)(1-P)\eta,(1-P)\eta\rangle}{\|(1-P)\eta\|^2}$$
for all $p\in \GM$. One checks that $L$ is a convex combination of $L_1$ and $L_2$. Since also $L_j\in\cC$, we obtain $L=L_1=L_2$ by the extreme property of $L$. Let $\lambda=\|P\eta\|^2$; then \eqref{e:pil} implies
$$\langle \pi(p)\eta,\lambda\eta\rangle
=\lambda\langle\pi(p)\eta,\eta\rangle
=\langle\pi(p)P\eta,P\eta\rangle
=\langle P\pi(p)\eta,P\eta\rangle
=\langle \pi(p)\eta,P\eta\rangle$$
for all $p\in \GM$. Therefore $P\eta=\lambda\eta$ since $\eta$ is a cyclic vector for $\pi$. Then $\lambda\in\{0,1\}$ since $P$ is a projection, a contradiction.

Next we show that $\pi$ is an irreducible representation. Suppose that $U\subseteq H$ is a closed $\pi$-invariant subspace and $P:H\to U$ the orthogonal projection. If $p\in\GM$ and $p^{\ti}=\pm p$, then
$$\pi(p)P=P\pi(p)P=\pm(P\pi(p)P)^*=\pm(\pi(p)P)^*=P\pi(P).$$
Consequently $P\in \cA'$ and hence $P\in\R$. Since $P$ is an orthogonal projection, we have $P\in\{0,1\}$, so $\pi$ is irreducible.

\smallskip {\sc Step 5:} Transition to $n\times n$ matrices.\\
We claim that $\cA$ is a prime algebra. Indeed, suppose $a\cA b=0$ for $a,b\in\cA$. If $b\neq0$, then there is a $u\in H$ with $bu\neq0$. Since $\pi$ is irreducible, the vector space $V=\cA bu$ is dense in $H$. Now $aV=0$ implies $aH=0$, i.e., $a=0$.

Since the $*$-center of $\cA$ equals $\R$, we have $\pi(\ceng)=\R$, so $\cA$ is generated by $\pi(\Xi_j)$ for $1\le j\le g$ as an $\R$-algebra with involution. Let $f\in\R\mx$ be a polynomial $*$-identity of $(\mat,{\ti})$. Then $f(p_1,\dots,p_k,p_1^{\ti},\dots,p_k^{\ti})=0$ for all $p_i\in\GM$. Therefore
$$f(\pi(p_1),\dots,\pi(p_k),\pi(p_1)^*,\dots,\pi(p_k)^*)=\pi(f(p_1,\dots,p_k,p_1^{\ti},\dots,p_k^{\ti}))=0,$$
so $f$ is a polynomial $*$-identity for $\cA$. By the $*$-version of Posner's theorem \cite[Theorem 2]{Row1} it follows that $\cA$ is central simple algebra of degree $n'\le n$ with involution $*$ and with $*$-center $\R$. Furthermore, the involution on $\cA$ is positive since it is a restriction of the adjoint involution on $\cB(H)$. By \cite[Theorem 1.2]{PS}, $\cA$ is $*$-isomorphic to one of
$$(\opm_{n'}(\R),\ti),\qquad (\opm_{n'}(\C),*),\qquad (\opm_{n'/2}(\HH),\si).$$
Since $\cA=\Phi(\GM)$, $\cA$ satisfies all polynomial $*$-identities of $(\mat,\ti)$. If $(\cA,*)\cong (\opm_{n'/2}(\HH),\si)$, then Proposition \ref{p:ts} implies $\frac{n'}{2}\le 2n$. Similarly, $(\cA,*)\cong (\opm_{n'}(\C),*)$ implies $2n'\le n$ by Remark \ref{r:ts}. Hence in all cases there exists a $*$-embedding of $(\cA,*)$ into $(\mat,\ti)$, so we can assume that $X_j:=\hat M_{\Xi_j}\in\mat$, $*=\ti$ and $\eta\in\R^d$. Since $L(\cQ)\subseteq\R_{\ge0}$, \eqref{e:pil} implies that $q(X)$ is positive semidefinite for all $q\in\cQ$, so $X\in K_{\cQ}$.

\smallskip {\sc Step 6:} Conclusion.\\
By \eqref{e:gns} we have
$$0\ge L(a)=\langle \overline a,\overline 1\rangle=\langle a(X,X^{\ti})\eta,\eta\rangle.$$
Therefore $a$ is not positive definite at $X\in K_{\cQ}$.
\end{proof}

\subsection{Putinar's theorem for trace polynomials}

Our final result in this section is Putinar's Positivstellensatz for trace polynomials, Theorem \ref{t:arch}. Our proof combines functional analytic techniques from the proof of Theorem \ref{t:arch1} with an algebraic commutative reduction.

\begin{lem}\label{l:archcen}
Let $\cQ\subset\STR$ be an archimedean cyclic quadratic module and $c\in\cen$. If $c|_{K_{\cQ}}>0$, then $c\in \cQ$.
\end{lem}

\begin{proof}
Assume $c\in \STR\setminus\cQ$. Steps 1--4 in the proof of Theorem \ref{t:arch1} work if we replace $\GM$ with $\TR$. Hence we obtain a Hilbert space $H$, a $*$-representation $\pi:\TR\to\cB(H)$ with $\pi(\cen)=\R$, and a cyclic unit vector $\eta\in H$ for $\pi$ such that the linear functional
$$L:\TR\to\R,\qquad p\mapsto \langle\pi(p)\eta,\eta\rangle$$
satisfies $L(\cQ)\subseteq \R_{\ge0}$ and $L(c)\le 0$. Let $\phi=\pi|_{\cen}:\cen\to\R$. By the proof of Theorem \ref{t:hom}, $\phi$ extends to a trace preserving $*$-homomorphism $\Psi:\TR\to\mat$. Let $X_j=\Psi(\Xi_j)\in\mat$ for $j=1,\dots,g$. Because $\pi(\cen)=\R$, we have $L|_{\cen}=\pi|_{\cen}$ and therefore $\tr(q(X))=\phi(\tr(q))\ge0$ for every $q\in\cQ$, so $X\in K_{\cQ}$ by Corollary \ref{c:scalcqm} and $c(X)=\phi(c)\le0$.
\end{proof}

\begin{thm}[Putinar's theorem for trace polynomials]\label{t:arch}
Let $\cQ\subset\STR$ be an archimedean cyclic quadratic module and $a\in\STR$. If $a|_{K_{\cQ}}\succ0$, then $a\in \cQ$.
\end{thm}

\begin{proof}
Let $\sigma_j=\tr(\wedge^j a)$ and assume $a|_{K_{\cQ}}\succ0$. Since $K_{\cQ}$ is compact, there exists $\varepsilon>0$ such that $(\sigma_j-\varepsilon)|_{K_{\cQ}} >0$ for all $1\le j\le n$. By Lemma \ref{l:archcen} we have $\sigma_j-\varepsilon\in\cQ$ for all $j$. Let $c_1,\dots,c_N$ be the generators of $\cen$ as an $\R$-algebra. Since $\cQ$ is archimedean, there exist $\rho_1,\dots,\rho_N\in\Q_{>0}$ such that $\rho_i-c_i^2\in\cQ$. Write
$$S=\left\{\sigma_j-\varepsilon,\rho_i-c_i^2\colon 1\le j\le n, 1\le i\le N\right\}$$
and let $Q$ be the quadratic module in $\cen[a]$ generated by $S\cup\TPc$. Clearly we have $Q\subset\cQ$ and $a|_{K_Q}>0$. Let $H_Q\subseteq \cen[a]$ be the subring of bounded elements with respect to $Q$, i.e., $b\in\cen[a]$ such that $\rho-b^2\in Q$ for some $\rho\in\Q_{>0}$.  Because $c_i$ generate $\cen$, we have $\cen\subseteq H_Q$. Since $H_Q$ is integrally closed in $\cen[a]$ by \cite[Section 6.3]{Bru} or \cite[Theorem 5.3]{Schw}, we also have $a\in H_Q$. Hence $Q$ is archimedean. If $P$ is an ordering in $\cen[a]$ containing $S\cup\TPc$, then $a\in P\setminus -P$ by Proposition \ref{p:ord}. Therefore $a\in Q\subset\cQ$ by Jacobi's representation theorem \cite[Theorem 5.4.4]{Mar}.
\end{proof}

%%%%%%%%%%%%%%%%%%%%%%%%%%%%%%%%%
%%%%%%%%%%%%%%%%%%%%%%%%%%%%%%%%%

\section{Examples}\label{s:exa}

In this section we collect some examples
and counterexamples pertaining to the results presented above.

\begin{exa}\label{exa:cenred}
Proposition \ref{p:scal} states that for every $s\in\STR$ there exists a finite set $S\subset\tr(\cQ^{\tr}_{\{s\}})$ such that $K_{\{s\}}=K_{S}$. Let us give a concrete example of such a set for $n=3$. Let $\sigma_j=\tr(\wedge^j s)$; using the Cayley-Hamilton theorem and the relations between $\sigma_j$ and $\tr(s^i)$ it is easy to check that
\begin{alignat*}{3}
\tr(s) &= \sigma_1, \\
\tr\big((s-\sigma_1)s(s-\sigma_1)\big) &= \sigma_1\sigma_2+3\sigma_3, \\
\tr\big((s^2-\sigma_1s+\sigma_2)s(s^2-\sigma_1s+\sigma_2)\big) &= \sigma_2\sigma_3, \\
\tr\big((s-\sigma_1-1)s(s-\sigma_1-1)\big) &= \sigma_1+4\sigma_2+3\sigma_3+\sigma_1\sigma_2. \\
\end{alignat*}
Denote these elements by $c_1,c_2,c_3,c_4\in\tr(\cQ^{\tr}_{\{s\}})$. We claim that $K_{\{s\}}= K_{\{c_1,c_2,c_3,c_4\}}$. It suffices to prove the inclusion $\supseteq$.

To simplify the notation let $s=s^\ti\in \opm_3(\R)$. If $s\nsucceq0$, then $\sigma_j<0$ for some $j$. If $\sigma_1<0$, then $c_1<0$. Hence assume $\sigma_1\ge0$. If $\sigma_2,\sigma_3$ are of opposite sign, then $c_3<0$. If $\sigma_1>0$ and $\sigma_2,\sigma_3\le0$, then $c_2<0$. Finally, if $\sigma_1=0$ and one of $\sigma_2,\sigma_3$ equals $0$, then $\sigma_3<0$ implies $c_2<0$ and $\sigma_2<0$ implies $c_4<0$.
\end{exa}

\begin{exa}\label{exa:det}
The denominator in Lemma \ref{l:rey} is unavoidable even if $f$ is a hermitian square or $f\in\R[\ulxi]$. For example, let $n=4$ and $a=\Xi_1-\Xi_1^{\ti}$. Then $\det(a)$ is a square in $\R[\ulxi]$. Suppose $\det(a)\in\Omega_4$, i.e.,
$$\det(a)=\sum_i\tr(h_{i1}h_{i1}^{\ti})\cdots \tr(h_{im_i}h_{im_i}^{\ti})$$
for some $h_{ij}\in\TR$. Because $a$ is independent of $\Xi_1+\Xi_1^{\ti}$ and $\Xi_j$ for $j> 1$, we can assume that $h_{ij}$ are polynomials in $a$ and $\tr(a^k)$ for $k\in\N$. Moreover, $\det(a)$ is homogeneous of degree 4 with respect to the entries of $a$, so $h_{ij}$ are of degree at most $2$ with respect to $a$. Finally,  $\tr(a)=0$, so we conclude that
$$\det(a)=\sum_i\tr\left((\alpha_i a^2+\beta_i\tr(a^2))^2\right),\qquad \alpha_i,\beta_i\in\R.$$
If $R$ is a real closed field containing $\R(\ulxi)$, then there exist $\lambda,\mu\in R$ that are algebraically independent over $\R$ such that
$$a=\begin{pmatrix}
0 & -\lambda & 0 & 0 \\
\lambda & 0 & 0 & 0 \\
0 & 0 & 0 & -\mu \\
0 & 0 & \mu & 0 \\
\end{pmatrix}$$
after an orthogonal basis change. Therefore
$$(\lambda\mu)^2=2\sum_i\left(
(\alpha_i\lambda^2+2\beta_i(\lambda^2+\mu^2))^2+
(\alpha_i\mu^2+2\beta_i(\lambda^2+\mu^2))^2
\right),$$
which is clearly a contradiction.
\end{exa}

\begin{exa}\label{exa:gm2}
Next we show that a traceless equivalent of Theorem \ref{t:posss} fails for $\gm_2$. A quadratic module $\cT\subseteq \SGM$ is a {\bf preordering} if $\cT\cap \ceng$ is closed under multiplication. For $S\subset\SGM$ let $\cQ_S$ and $\cT_S$ denote the quadratic module and preordering, respectively, generated by $S$. For example, $\cT_\emptyset=\cQ_\emptyset$ is the set of sums of hermitian squares in $\GM$.

Fix $n=2$ and let
$$s=\Xi_1+\Xi_1^{\ti},\qquad a=\Xi_1-\Xi_1^{\ti},\qquad f=[s^2,a][s,a].$$
Since $f=\tr(s) [s,a]^2$, it is clear that $f|_{K_{\{s\}}}\succeq0$. We will show that there do not exist $t_1,t_2\in\cT_S$ and $k\in\N$ such that
\begin{equation}\label{e:cex}
ft_1=t_1f=f^{2k}+t_2.
\end{equation}

It clearly suffices to assume $g=1$. If $R$ is a real closed field containing $\R(\ulxi)$, then after diagonalizing $s$ we may assume that
$$s=\begin{pmatrix}\lambda_1 & 0 \\ 0 & \lambda_2\end{pmatrix},\qquad
a=\begin{pmatrix}0 & -\mu \\ \mu & 0\end{pmatrix}$$
for some $\lambda_1,\lambda_2,\mu\in R$ that are algebraically independent over $\R$. Let $h\in\gm_2$ be homogeneous of degree $(d,e)$ with respect to $(s,a)$. Then it is not hard to check that
\begin{equation}\label{e:sq1}
hs h^{\ti}=\mu^{2e}\begin{pmatrix} \lambda_1\tilde{h}(\lambda_1,\lambda_2)^2 & 0 \\
0 & \lambda_2\tilde{h}(\lambda_2,\lambda_1)^2 \end{pmatrix}
\end{equation}
for some homogeneous polynomial $\tilde{h}\in\R[y_1,y_2]$ of degree $d$. Therefore
$$\sum_i h_ish_i^{\ti}\in C_2 \quad \Rightarrow \quad h_i=0 \ \forall i,$$
so we deduce that
\begin{equation}\label{e:tq}
\cT_{\{s\}}=(\cT_\emptyset\cap C_2) \cdot\cQ_{\{s\}}.
\end{equation}
Now suppose that \eqref{e:cex} holds for some $t_1,t_2\in\cT_S$ and $k\in\N$. Since $f$ is homogeneous of degree $(5,2)$ with respect to $(s,a)$, $t_1,t_2$ can be taken homogeneous as well. Then $t_2$ is of degree $(10k,4k)$ and $t_1$ is of degree $(10k-5,4k-2)$. In particular, the total degrees of $t_1$ and $t_2$ are odd and even, respectively, so by \eqref{e:tq} we conclude that $t_2\in \cT_\emptyset$ and $t_1$ is of the form $\sum_ih_ish_i^{\ti}$ for $h_i\in\gm_2$. Hence \eqref{e:sq1} implies
$$t_1=\mu^{4k-2}\begin{pmatrix} \lambda_1\sum_i\tilde{h}_i(\lambda_1,\lambda_2)^2 & 0 \\
0 & \lambda_2\sum_i\tilde{h}_i(\lambda_2,\lambda_1)^2 \end{pmatrix}$$
for some homogeneous $\tilde{h}_i \in\R[y_1,y_2]$. Now
$$f=(\lambda_1+\lambda_2)(\lambda_1-\lambda_2)^2\mu^2$$
implies
\begin{equation}\label{e:sq2}
ft_1=(\lambda_1+\lambda_2)\mu^{4k}(\lambda_1-\lambda_2)^2\begin{pmatrix} \lambda_1\sum_i\tilde{h}_i(\lambda_1,\lambda_2)^2 & 0 \\
0 & \lambda_2\sum_i\tilde{h}_i(\lambda_2,\lambda_1)^2 \end{pmatrix}.
\end{equation}
The nonempty set
$$\left\{X\in\opm_2(\R)\colon \det(s(X))<0 \text{ and } \tr(s(X))>0\right\}$$
is open in the Euclidean topology, so by \eqref{e:sq2} there exists $X\in\opm_2(\R)$ such that $(st_1)(X)$ is nonzero and indefinite. However, this contradicts $ft_1=f^{2k}+t_2\in\cT_\emptyset$.
\end{exa}

\begin{exa}\label{exa:cent}
Here we show that the element $t_1$ in Theorem \ref{t:posss} cannot be chosen central in general. Let $n=2$, $s=\frac12(\Xi_1+\Xi_1^{\ti})$ and $S=\{\tr(s)^3,\det(s)^3\}$. Suppose $t_1s=s^{2k}+t_2$ for some $k\in\N$ and $t_1,t_2\in\cT^{\tr}_S$ and $t_1\in T_2$. Let
$$\Psi:\TT\to\opm_2(\R[\zeta]),\qquad \Xi_1\mapsto \begin{pmatrix}\zeta & 0 \\0 & 1\end{pmatrix}.$$
Since $(\zeta+1)^3=\zeta^3+(\frac32\zeta+1)^2+\frac34\zeta^2$, we conclude that $\Psi(t_1)$ belongs to the commutative preordering generated by $\zeta^3$ and $\Psi(t_2)$ belongs to the matricial preordering generated by $\zeta^3$ in $\opm_2(\R[\zeta])$. But then $\Psi(t_1)\Psi(s)=\Psi(s)^{2k}+\Psi(t_2)$ contradicts \cite[Example 4]{Cim}.
\end{exa}

\begin{exa}\label{exa:intro}
Let $f=5 \Tr(\Xi_1\Xi_1^{\ti})-2 \Tr(\Xi_1)(\Xi_1+\Xi_1^{\ti})\in\TT_2$. We will show that $f$ is totally positive and write it as a sum of hermitian squares in $\usa_2$. Write $\Xi=\Xi_1=(\xi_{\imath\jmath})_{\imath\jmath}$ and let
$$u=\begin{pmatrix}\eta_1 & \eta_2\end{pmatrix},\qquad
v=\begin{pmatrix}
\xi_{22}\eta_1 & \xi_{21}\eta_2 & \xi_{12}\eta_2 & \xi_{11}\eta_1 &
\xi_{22}\eta_2 & \xi_{21}\eta_1 & \xi_{12}\eta_1 & \xi_{11}y_2
\end{pmatrix}.$$
Then $u f u^{\ti}$ can be viewed as a quadratic form in $v$ and
$$u f u^{\ti}=v G_{\alpha}v^{\ti},\qquad G_{\alpha}=
\begin{pmatrix}
5 & \alpha & \alpha & -2 & 0 & 0 & 0 & 0 \\
\alpha & 5 & 0 & -\alpha-2 & 0 & 0 & 0 & 0 \\
\alpha & 0 & 5 & -\alpha-2 & 0 & 0 & 0 & 0 \\
-2 & -\alpha-2 & -\alpha-2 & 1 & 0 & 0 & 0 & 0 \\
0 & 0 & 0 & 0 & 1 & -\alpha-2 & -\alpha-2 & -2 \\
0 & 0 & 0 & 0 & -\alpha-2 & 5 & 0 & \alpha \\
0 & 0 & 0 & 0 & -\alpha-2 & 0 & 5 & \alpha \\
0 & 0 & 0 & 0 & -2 & \alpha & \alpha & 5
\end{pmatrix}$$ 
for $\alpha\in\R$. Observe that $G_{\alpha}$ is positive semidefinite if and only if $-\frac72\le \alpha\le -\frac52$. Hence $f$ is indeed totally positive and a sum of hermitian squares in $\opm_2(\R[\ulxi])$. By diagonalizing $G_{\alpha}$ at $\alpha=-\frac52$ we obtain
$$f=\frac52 (\xi_{12}-\xi_{21})^2
+\frac12 \tilde{H}_2\tilde{H}_2^{\ti}+\frac12 \tilde{H}_3\tilde{H}_3^{\ti},$$
where
$$\tilde{H}_2=\begin{pmatrix}
\xi_{12}+\xi_{21} & \xi_{22}-\xi_{11} \\ \xi_{11}-\xi_{22} & \xi_{12}+\xi_{21}
\end{pmatrix},\qquad
\tilde{H}_3=\begin{pmatrix}
2(\xi_{12}+\xi_{21}) & \xi_{11}-3\xi_{22} \\ \xi_{22}-3\xi_{11} & 2(\xi_{12}+\xi_{21})
\end{pmatrix}.$$
Note that while $\tilde{H}_2\tilde{H}_2^{\ti},\tilde{H}_3\tilde{H}_3^{\ti}\in\Sym\TT_2$, we can compute $\cR_2(\tilde{H}_2)=\cR(\tilde{H}_3)=0$, so $\tilde{H}_2,\tilde{H}_3\notin\TT_2$. However, if we set
$$H_1=\Xi-\Xi^{\ti},\qquad
H_2=\Xi\Xi^{\ti}-\Xi^{\ti}\Xi,\qquad
H_3=\Xi^2-2\Xi\Xi^{\ti}+2\Xi^{\ti}\Xi-(\Xi^{\ti})^2,$$
then
$$H_1H_1^\ti=(\xi_{12}-\xi_{21})^2,\qquad
H_2H_2^{\ti}=(\xi_{12}-\xi_{21})^2\tilde{H}_2\tilde{H}_2^{\ti},\qquad
H_3H_3^{\ti}=(\xi_{12}-\xi_{21})^2\tilde{H}_3\tilde{H}_3^{\ti}$$
and so
$$f=\frac52 H_1H_1^{\ti}+\frac12 H_1^{-1}H_2H_2^{\ti}H_1^{-\ti}+\frac12 H_1^{-1}H_3H_3^{\ti}H_1^{-\ti}.$$

\end{exa}

%%%%%%%%%%%%%%%%%%%%%%%%%%%%%%%%%
%%%%%%%%%%%%%%%%%%%%%%%%%%%%%%%%%

\begin{appendices}

\section{Constructions of the Reynolds operator}\label{a:rey}

In this appendix we describe a few more ways of constructing $\rey$ for the action of $\ort$ on $\matpoly$ defined in Subsection \ref{ss:gentr}. We refer to \cite{Stu} for algorithms for finite group actions. Let $\rey':\R[\ulxi]\to\cen$ be the restriction of $\rey:\matpoly\to\TR$, i.e., the Reynolds operator for the action of $\ort$ on $\R[\ulxi]$ given by
$$f^u=f(u\Xi_1u^{\ti},\dots,u\Xi_gu^{\ti})$$
for $f\in\R[\ulxi]$ and $u\in\ort$.

\subsection{Computing \texorpdfstring{$\rey'$}{Rn'}}

We start by describing two ways of obtaining $\rey'$.

\subsubsection{First method}

We follow \cite[Section 4.5.2]{DK} to present an algorithm for computing $\rey'$. We define a linear map $c\in \R[\ort]^*$:
$$
c(r)=\frac{\rm d}{{\rm d}t}\frac{\rm d}{{\rm d}s}\sum_{i,j=1}^n
r\big((1+se_{ij})(1+te_{ij})\big)-r\big((1+se_{ij})(1+te_{ji})\big)\bigg|_{s=t=0},
$$ 
where $e_{ij}$, $1\leq i,j\leq n$, denote the standard matrix units in $\mat$. (In fact, $c$ equals the Casimir operator of the Lie algebra $\mathfrak{o}_n$ of skew symmetric matrices of the group $\ort$ up to a scalar multiple.) 
For example, if $n=2$, then 
\begin{eqnarray*}
	c(u_{11}u_{22})&=&\frac{\rm d}{{\rm d}t}\frac{\rm d}{{\rm d}s}\Big(
	\big((1+se_{12})(1+te_{12})\big)_{11}\big((1+se_{12})(1+te_{12})\big)_{22}-\\
	&&-\big((1+se_{12})(1+te_{21})\big)_{11}\big((1+se_{12})(1+te_{21})\big)_{22}+\\
	&&+\big((1+se_{21})(1+te_{21})\big)_{11}\big((1+se_{21})(1+te_{21})\big)_{22}-\\
	&&-\big((1+se_{21})(1+te_{12})\big)_{11}\big((1+se_{21})(1+te_{12})\big)_{22}\Big)\bigg|_{s=t=0}\\
	&=&\frac{\rm d}{{\rm d}t}\frac{\rm d}{{\rm d}s}(1-st+1-st)\bigg|_{s=t=0}\\
	&=&-2.
\end{eqnarray*}
For $f\in \R[\ulxi]$, $u=(u_{\imath\jmath})_{\imath\jmath}\in\ort$, write $f^u$ as $\sum_i f_i\mu_i$, 
where $f_i$ are linearly independent polynomials in the variables $\xi_{j\imath\jmath}$ and $\mu_i$ are polynomials in the variables $u_{\imath\jmath}$. Define 
$$\tilde c(f)=\sum f_ic(\mu_i).$$
Find the monic polynomial $p$ of smallest degree such that $p(\tilde c)(f)=0$. 
If $p(0)\neq 0$, set $\rey''(f)=f$. 
If $p(0)=0$, write $p(t)=tq(t)$ and define 
$\rey''(f)=q(0)^{-1}q(\tilde c(f))$.
By \cite[Proposition 4.5.17]{DK}, $\rey''$ defines the Reynolds operator for the action of $\sort$ on $\R[\ulxi]$. 
Since $\ort/\sort \cong\Z/2\Z$, setting $\rey'(f)=\frac{1}{2}(\rey''(f)+\rey''(f)^v)$, where $v$ is an arbitrary element in  $\ort\setminus\sort$, we obtain the Reynolds operator for the action of $\ort$ on $\R[\ulxi]$.

\subsubsection{Second method}

Here we mention another way of computing the Reynolds operator $\rey'$ in terms of an integral. 
This approach is based on the way the invariants of $\ort$ for the action on $\R[\ulxi]$ were described by Procesi in \cite{Pro}. Let $f\in \R[\ulxi]$. We first multihomogenize $f$ as a function $f:\mat^g\to \R$, then multilinearize its homogeneous components $f_i$ and view $f_i$ as an element $\overline f_i\in(\mat^{\otimes d_i})^*$ for $d_i=\deg(f_i)$. 
Since $\mat\cong V^*\otimes V$ for a $n$-dimensional vector space $V$ on which $\ort$ acts naturally, and 
$V^*$ is isomorphic as a $\ort$-module to $V$, $\overline f_i$ can be seen as an element $\tilde f_i\in V^{\otimes 2d_i}$. 
The monomial $\xi_{1\imath_1\jmath_1}\cdots \xi_{d\imath_d\jmath_d}$ corresponds to the element 
$$e_{\imath_1}\otimes e_{\jmath_1}\otimes \cdots\otimes e_{\imath_d}\otimes e_{\jmath_d},$$
where $e_i$, $1\leq i\leq n$, is an orthonormal basis of $V$.
Then we can compute $\rey'(\overline f_i)$ by integrating the function $u\mapsto (\tilde f_i)^u$ over $\ort$. To obtain $\rey'(f_i)$ we need to restitute $\rey'(\overline f_i)$ and multiply the result by a suitable integer. Finally, $\rey'(f)=\sum \rey'(f_i)$.

\subsection{From \texorpdfstring{$\rey'$}{Rn'} to \texorpdfstring{$\rey$}{Rn}}

Once we have $\rey'$, we can compute $\rey$ as follows.

\subsubsection{First method}

Let $f\in \matpoly$. We can assume that $f$ is independent of $\Xi_g$. Let us compute $\rey'(\tr(f\Xi_g))$. Since this is an invariant,  linear in $\Xi_g$, it has the form 
$$\rey'(\tr(f \Xi_g))=\tr\Big(f_0 \Xi_g\Big)$$
for some $f_0\in\TR$. Here we used the fact that $\tr(h\Xi_g^{\ti})=\tr(h^{\ti}\Xi_g)$ for $h\in\TR$. We define $\rey(f)=f_0$. We claim that $\rey:\matpoly\to \TR$ is the Reynolds operator. We have $\rey(f)=f$ for $f\in \TR$ since $\rey'(\tr(f\Xi_g))=\tr(f\Xi_g)$ as $\tr(f\Xi_g)\in \cen$. For $u\in \ort$, $f\in \matpoly$ we have 
$$\tr((f^u)\Xi_g)=\tr((f\Xi_g)^u)=\tr(f\Xi_g)^u,$$ 
where the first equality follows as $\Xi_g$ is an invariant and the second one since $\tr$ is linear. Thus, $\rey(f^u)=\rey(f)$, and $\rey$ is the Reynolds operator.

\subsubsection{Second method}

If we have $\rey'$, then the Reynolds operator can also be computed by expressing an element $f\in \matpoly$ as an $\R(\ulxi)$-linear combination of $\Xi_1^i\Xi_1^{\ti j}$, $0\leq i,j\leq n-1$.  Note that these elements are linearly independent in $\opm_n(\R(\ulxi))$ as there exists $X\in \mat$ such that $X^iX^{\ti j}$, $0\leq i,j\leq n-1$, are linearly independent. We denote $\Xi_1^i\Xi_1^{\ti j}$, $0\leq i,j\leq n-1$, by $y_1,\dots,y_{n^2}$. Let $c$ be a $n^2$-normal (i.e., multilinear and alternating in the first $n^2$-variables) central polynomial of $\mat$ in $2n^2-1$ variables. 
(See e.g.~ \cite[Section 1.4]{Row2} for the construction of such polynomials and for the proofs of their properties mentioned below.)
Since $y_1,\dots,y_{n^2}$ are independent we can find $y_{n^2+1},\dots, y_{2n^2-1}\in \GM$ such that 
$$0\neq c(y_1,\dots,y_{n^2},y_{n^2+1},\dots,y_{2n^2-1})=z\in\ceng.$$
If $g\ge n^2$, we can take $y_{n^2+1}=\Xi_2,\dots,y_{2n^2-1}=\Xi_{n^2}$.
Then $f$ can be written as follows %(see e.g.~ \cite[Lemma 1.4.20]{Row}):
\begin{equation}\label{e:xy}
f=\sum_{i=1}^{n^2}(-1)^{i-1} z^{-1}c(f,y_1,\dots,y_{i-1},y_{i+1},\dots,y_d)y_i.
\end{equation}
Let $z_i(f)=\rey'(c(f,y_1,\dots,y_{i-1},y_{i+1},\dots,y_d))$. 
We define 
$$
\rey(f)=\sum_{i=1}^{n^2}(-1)^{i+1}z^{-1} z_i(f)y_i.
$$
If $f\in \TR$, then the coefficients in the expression \eqref{e:xy} are already in $\cen$, so in this case $\rey(f)=f$. Note that $\rey'(z_i(f^u))=\rey'(z_i(f))$ for $u\in \ort$ and $f\in \matpoly$. Therefore $\rey(f^u)=\rey(f)$ and $\rey:\matpoly\to \TR$ is the Reynolds operator.

\section{How not to prove the extension Theorem \ref{t:hom}}\label{a:ps}

One might attempt to prove Theorem \ref{t:hom} using geometric invariant theory of Lie groups \cite{PS1,Bro,CKS}. Here we explain why this approach fails.

Let $G$ be a compact Lie group with an orthogonal representation on $W=\R^N$. The invariant ring $\R[W]^G$ is a finitely generated $\R$-algebra; let $p_1,\dots,p_m$ be its generators. Let $(\cdot,\cdot)$ denote a $G$-invariant inner product on $W$ and its dual on $W^*$. Since the differentials ${\rm d}p_i:W\to W^*$ are $G$-equivariant, we have $({\rm d}p_i,{\rm d}p_j)\in \R[W]^G$. Finally let
$$\Grad=\big(({\rm d}p_i,{\rm d}p_j)\big)_{i,j}\in\opm_m(\R[W]^G).$$

The following theorem is a reformulation of the celebrated Procesi-Schwarz theorem \cite[Theorem 0.10]{PS1} and is essentially due to Schrijver \cite{Scr0} (see also \cite{Scr1}). We thank M. Schweighofer for drawing our attention to Schrijver's  work.

\begin{thm}[Procesi-Schwarz]\label{t:Bps}
Let $\phi:\R[W]^G\to R$ be an $\R$-algebra homomorphism into a real closed field $R\supseteq\R$. Then the following are equivalent:
\begin{enumerate}[label={\rm(\roman*)}]
\item $\phi$ extends to an $\R$-algebra homomorphism $\varphi:\R[W]\to R$;
\item $\phi(\R[W]^G\cap \sum \R[W]^2)\subseteq R_{\ge0}$;
\item $\phi(\Grad) \in \opm_m(R)$ is positive semidefinite.
\end{enumerate}
\end{thm}

\begin{proof}
While (i)$\Rightarrow$(ii) and (i)$\Rightarrow$(iii) are straightforward, (iii)$\Rightarrow$(ii) is involved and proved in \cite[Theorem 0.10]{PS1} and \cite[Subsection 2.7]{CKS}. Hence we are left with (ii)$\Rightarrow$(i). Without loss of generality we can assume that $\phi(\R[W]^G)$ generates $R$ as a field.

First we observe that if $\cR:\R[W]\to\R[W]^G$ is the Reynolds operator for the action of $G$, then $\cR(\sum \R[W]^2)\subseteq \sum \R[W]^2$. Note that since $G$ acts linearly on $W$, the action of $G$ on $\R[W]$ does not increase the degree of polynomials. Since $G$ is compact, $\cR$ is given by the integration formula $\cR(f)=\int_G f^g \, d\mu(g)$, where $\mu$ is the normalized left Haar measure. If $f\in\R[W]$ is of degree $d$, then $\cR(f^2)$ is a limit of sums of squares of degree $2d$. Using Carath\'eodory's theorem \cite[Theorem I.2.3]{Bar} it is easy to see that the cone of sums of squares in $\R[W]$ of degree at most $2d$ is closed in the space of polynomials of degree at most $2d$ (cf. \cite[Section 4.1]{Mar}). Hence we conclude that $\cR(f^2)$ is indeed a sum of squares in $\R[W]$.

Now let $T\subseteq \R[W]$ be the preordering generated by $\phi^{-1}(R_{\ge0})$. We claim that $-1\notin T$. Otherwise $-1=s_0+\sum_{i\ge1}s_it_i$ for some $s_i\in\sum \R[W]^2$ and $t_i\in \phi^{-1}(R_{\ge0})$. By applying $\cR$ we get
\begin{equation}\label{e:b1}
-1=\cR(s_0)+\sum_{i\ge1}\cR(s_i)t_i.
\end{equation}
By the above observation we have $\cR(s_i)\in \R[W]^G\cap \sum \R[W]^2$, so \eqref{e:b1} implies $$-1=\phi(-1)=\phi\left(\cR(s_0)\right)+\sum_{i\ge1}\phi\left(\cR(s_i)\right)\phi(t_i)\ge0,$$
a contradiction. Therefore we can extend $T$ to an ordering $P\subset \R[W]$, which gives rise to a homomorphism $\varphi_0:\R[W]\to R_0$, where $R_0$ is the real closure of the ordered field of fractions of $\R[W]/(P\cap -P)$. Since
$$\ker \phi\subseteq T\cap-T\subseteq P\cap -P = \ker \varphi_0,$$
we see that $\varphi_0$ extends $\phi$ and hence $R\subseteq R_0$. By the Artin-Lang homomorphism theorem \cite[Theorem 4.1.2]{BCR} there exists a homomorphism $\varphi:\R[W]\to R$ satisfying $\ker\varphi=\ker\varphi_0$. Thus $\varphi$ extends $\phi$.
\end{proof}

Theorem \ref{t:Bps} can be used to prove the following weakened version of Theorem \ref{t:hom}.

\begin{cor}\label{c:B}
An $\R$-algebra homomorphism $\phi:\cen\to \R$ extends to an $\R$-algebra homomorphism $\varphi:\R[\ulxi]\to \R$ if and only if $\phi(\TPc)\subseteq \R_{\ge0}$.
\end{cor}

Let us outline the proof. Let $p_1,\dots,p_m$ be generators of the $\R$-algebra $\cen=\R[\ulxi]^{\ort}$. Their differentials ${\rm d}p_i:\mat^g\to(\mat^g)^*$ are $\ort$-equivariant maps. Since we can identify $(\mat^g)^*$ with $(\mat^*)^g$, and the $\ort$-equivariant polynomial maps $\mat^g\to\mat$ are precisely trace polynomials \cite[Theorems 7.1 and 7.2]{Pro}, we have ${\rm d}p_i\in \TR^g$. On $\mat^g$ there is an $\ort$-invariant inner product
$$(X,Y)=\tr\left(\sum_j X_jY_j^\ti\right).$$
Finally, let $\Grad=(({\rm d}p_i,{\rm d}p_j))_{i,j}\in \opm_m(\cen)$.

Let $\phi:\cen\to \R$ be an $\R$-algebra homomorphism. Then Theorem \ref{t:Bps} implies that $\phi$ extends to an $\R$-algebra homomorphism $\varphi:\R[\ulxi]\to \R$ if and only if $\phi(\Grad)\in \opm_m(\R)$ is positive semidefinite. To prove the non-trivial direction in Corollary \ref{c:B} it therefore suffices to show the following.

\begin{lem}\label{l:B}
If $\phi(\TPc)\subseteq \R_{\ge0}$, then $\alpha^\ti\phi(\Grad)\alpha\ge0$ for all $\alpha\in\R^m$.
\end{lem}

\begin{proof}
Denote $h_{ij}=({\rm d} p_i)_j\in\TR$. If $\alpha=(\alpha_i)_i\in \R^m$, then
\begin{align*}
\alpha^\ti \Grad \alpha
&=\sum_{i_1,i_2} \alpha_{i_1}\alpha_{i_2}({\rm d}p_{i_1},{\rm d}p_{i_2}) \\
&=\sum_{i_1,i_2} \alpha_{i_1}\alpha_{i_2}\tr\left(\sum_jh_{i_1j}h_{i_2j}^\ti\right) \\
&=\sum_j\tr\left(\sum_{i_1,i_2} \alpha_{i_1}h_{i_1j}\alpha_{i_2}h_{i_2j}^\ti\right) \\
&=\sum_j\tr\left(\left(\sum_i\alpha_i h_{ij}\right)\left(\sum_i\alpha_i h_{ij}\right)^\ti \right).
\end{align*}
Therefore $\phi(\tr(hh^\ti))\ge0$ for all $h\in\TR$ implies that $\phi(\Grad)$ is positive semidefinite.
\end{proof}

From the proof of Lemma \ref{l:B} we see that to prove Theorem \ref{t:hom} using the Procesi-Schwarz theorem, one would need to extend the chain of equivalences in Theorem \ref{t:Bps} with the condition

\begin{enumerate}[label={\rm(\roman*')}]
\setcounter{enumi}{2}
\item $\alpha^\ti\phi(\Grad)\alpha\ge0$ for all $\alpha\in \R^m$.
\end{enumerate}

However, (iii')$\Rightarrow$(iii) fails in our context.

\def\mg{\operatorname{O}_2(\R)}

\begin{exa}
Let $\mg$ act on $\opm_2(\R)$ by conjugation, i.e., $n=2$ and $g=1$ in the setting of this paper. If $\Xi=\Xi^{(2)}$ is a generic $2\times 2$ matrix
and
$$y_1=\tr(\Xi),\qquad y_2=\tr(\Xi^2),\qquad y_3=\tr(\Xi\Xi^\ti),$$
then $\R[\ulxi]^{\mg}=\R\left[y_1,y_2,y_3\right]$ (see e.g. \cite{ADS}; algebraic independence follows from the Jacobian criterion). For this choice of generators we have
$$H=\begin{pmatrix}
2 & 2y_1 & 2y_1 \\
2y_1 & 4y_3 & 4y_2 \\
2y_1 & 4y_2 & 4y_3
\end{pmatrix}.$$
Let $R$ be the real closure of the rational function field $\R(\ve)$ endowed with the ordering $0<\ve<\alpha$ for every $\alpha\in\R_{>0}$. Consider the $\R$-algebra homomorphism
$$\phi:\R[y_1,y_2,y_3]\to R,\qquad
y_1\mapsto \frac{\ve}{2},\qquad
y_2\mapsto 0,\qquad
y_3\mapsto \frac18 \ve^2\left(1+\sqrt{1+4\ve^2}-2\ve\right).$$
For $\alpha=(\alpha_1,\alpha_2,\alpha_3)\in R^3$ we have
\begin{equation}\label{e:bh}
\begin{aligned}
\frac{\alpha^\ti \phi(\Grad)\alpha}{2}
&=\alpha_1^2+\ve\alpha_1(\alpha_2+\alpha_3)+\frac{\ve^2(1+\sqrt{1+4\ve^2}-2\ve)}{4}(\alpha_2^2+\alpha_3^2) \\
&=\left(\alpha_1+\frac{\ve(\alpha_2+\alpha_3)}{2}\right)^2 + \frac{\ve^2(\sqrt{1+4\ve^2}-2\ve)}{4}\left(\alpha_2-\frac{\alpha_3}{\sqrt{1+4\ve^2}-2\ve}\right)^2- \ve^3\alpha_3^2.
\end{aligned}
\end{equation}
If $\alpha_2,\alpha_3\in \R$ and $\alpha_1\neq0$ or $\alpha_3\neq-\alpha_2$, then
$$\left(\alpha_1+\frac{\ve(\alpha_2+\alpha_3)}{2}\right)^2>\ve^3\alpha_3^2;$$
and if $\alpha_3=-\alpha_2\in\R\setminus\{0\}$, then
$$\frac{\ve^2}{4}\left(\sqrt{1+4\ve^2}-2\ve\right)\left(\alpha_2+\frac{\alpha_2}{\sqrt{1+4\ve^2}-2\ve}\right)^2
=\frac{\ve^2}{2}\left(1+\sqrt{1+4\ve^2}\right)\alpha_2^2>\ve^3\alpha_2^2.$$
Therefore $\alpha^\ti \phi(\Grad)\alpha > 0$ for all $\alpha\in\R^3\setminus\{0^3\}$. On the other hand, from \eqref{e:bh} it is clear that we can choose $\alpha\in R^3$ such that $\alpha^\ti \phi(H)\alpha = -\ve^3<0$, so $\phi(H)$ is not positive semidefinite.
\end{exa}
	
\end{appendices}

%%%%%%%%%%%%%%%%%%%%%%%%%%%%%%%%%%%%%%%%%%%%%%%%%%%%%%%%%%%%%%%%%%%%%%%%%%%%
%%%%%%%%%%%%%%%%%%%%%%%%%%%%%%%%%%%%%%%%%%%%%%%%%%%%%%%%%%%%%%%%%%%%%%%%%%%%

\end{document}